\documentclass[10pt,reqno]{amsart}
\usepackage{amsmath}
\usepackage{amsthm}
\usepackage{amssymb}
\usepackage{amsfonts, mathdots}
\usepackage{graphicx}
\usepackage{latexsym}
\usepackage{times}
\usepackage{fancyhdr}
\usepackage{url}
\usepackage{multicol}
\usepackage{cite}
\usepackage[colorlinks, urlcolor=blue]{hyperref}
\usepackage{mathrsfs}
\usepackage[usenames,dvipsnames]{color} 
\usepackage[font={small}]{caption}

\usepackage{tikz,ifthen}
\usetikzlibrary{intersections}
\usetikzlibrary{calc}
\usetikzlibrary{decorations.pathreplacing}
\usetikzlibrary{positioning}
\usetikzlibrary{arrows,shapes,backgrounds,3d}
\usepackage{fix-cm}
\usetikzlibrary{decorations.pathreplacing,shapes,snakes}
\tikzstyle{box} = [rectangle, minimum width=1cm, minimum height=3cm,text centered, draw=black, fill=red!10]

\graphicspath{{./GRAPHICS/}}
\graphicspath{{./Graphs/}}

\usepackage{tikz}

\usepackage[font={footnotesize,it}]{caption, subcaption}

\usepackage{float}
\usetikzlibrary{decorations.pathreplacing}
\usetikzlibrary{positioning}
\usetikzlibrary{calc}
\usetikzlibrary{arrows,shapes,backgrounds,3d}
\usepackage{fix-cm}
\usetikzlibrary{decorations.pathreplacing,shapes,snakes}
\tikzstyle{box} = [rectangle, minimum width=1cm, minimum height=3cm,text centered, draw=black, fill=red!10]

\usepackage{mathabx} 
\renewcommand{\Re}[1]{\operatorname{Re} #1 }

\newcommand{\R}{\mathbb{R}}

\newcommand{\C}{\mathbb{C}}

\newcommand{\norm}[1]{\left\| #1 \right\|}

\let\oldenumerate=\enumerate
	\def\enumerate{
	\oldenumerate
	\setlength{\itemsep}{5pt}
	}
\let\olditemize=\itemize
	\def\itemize{
	\olditemize
	\setlength{\itemsep}{5pt}
	}
\newtheorem{theorem}[equation]{Theorem}

\newtheorem{lemma}[equation]{Lemma}

\numberwithin{equation}{section}

\theoremstyle{definition}

\newtheorem{remark}[equation]{Remark}

\allowdisplaybreaks

\newcommand{\Li}{\operatorname{Li}}
\newcommand{\floor}[1]{\lfloor #1 \rfloor}

\usetikzlibrary{calc,decorations.markings}
\begin{document}

\title[The Prime Number Theorem]{The Prime Number Theorem\\ as a Capstone in a Complex Analysis Course}
	\author{Stephan Ramon Garcia}
	\address{Department of Mathematics, Pomona College, 610 N. College Ave., Claremont, CA 91711} 
	\email{stephan.garcia@pomona.edu}
	\urladdr{\url{http://pages.pomona.edu/~sg064747}}

\thanks{SRG supported by NSF Grant DMS-1800123.}



\begin{abstract}
We present a detailed proof of the prime number theorem suitable for a typical undergraduate- or graduate-level complex analysis course.  Our presentation is particularly useful for any instructor who seeks to use the prime number theorem for a series of capstone lectures, a scaffold for a series of guided exercises, or as a framework for an inquiry-based course.  We require almost no knowledge of number theory, for our aim is to make a complete proof of the prime number theorem widely accessible to complex analysis instructors.  In particular, we highlight the potential pitfalls and subtleties that may catch the instructor unawares when using more terse sources.
\end{abstract}
\maketitle

\section{Introduction}
The prime number theorem is one of the great theorems in mathematics.  It unexpectedly connects the discrete and the continuous with the elegant statement
\begin{equation*}
\lim_{x\to\infty} \frac{\pi(x)}{x/\log x} =1,
\end{equation*}
in which $\pi(x)$ denotes the number of primes at most $x$.  The original proofs, and most modern proofs,
make extensive use of complex analysis.
Our aim here is to present, for the benefit of complex analysis instructors, a complete
proof of the prime number theorem suitable either as a sequence of capstone lectures at the end of the term, a scaffold for a series of exercises,
or a framework for an entire inquiry-based course.  We require almost no knowledge of number theory.  In fact, our aim is to make
a detailed proof of the prime number theorem widely accessible to complex analysis instructors of all stripes.

Why does the prime number theorem belong in a complex-variables course?  At various stages, the proof utilizes
complex power functions, the complex exponential and logarithm, power series, Euler's formula, analytic continuation, the Weierstrass $M$-test, locally uniform convergence, zeros and poles, residues, Cauchy's theorem, Cauchy's integral formula, Morera's theorem, and much more. 
Familiarity with limits superior and inferior is needed toward the end of the proof,
and there are plenty of inequalities and infinite series.

The prime number theorem is one of the few landmark mathematical results whose proof 
is fully accessible at the undergraduate level.  Some epochal theorems, like the Atiyah--Singer index theorem,
can barely be stated at the undergraduate level.  Others, like Fermat's last theorem, 
are simply stated, but have proofs well beyond the undergraduate curriculum.
Consequently, the prime number theorem provides a unique opportunity for students
to experience a mathematical capstone that draws upon the entirety of a course and which culminates in the complete
proof of a deep and profound result that informs much current research.  In particular, students gain an
understanding of and appreciation for the Riemann Hypothesis, 
perhaps the most important unsolved problem in mathematics.
One student in the author's recent class proclaimed,
``I really enjoyed the prime number theorem being the capstone of the course. It felt rewarding to have a large proof of an important theorem be what we were working up towards as opposed to an exam.''  Another added,
``I enjoyed the content very much\ldots I was happy I finally got to see a proof of the result.''

Treatments of the prime number theorem in complex analysis texts, if they appear at all, are often terse
and nontrivial to expand at the level of detail needed for our purposes.
For example, the standard complex analysis texts 
\cite{Boas,Dangelo,Krantz, Marsden, Narasimhan, Rudin, Saff, Sarason, Stewart}
do not include proofs of the prime number theorem, although they distinguish themselves 
in many other respects.  A few classic texts 
\cite{Ahlfors, Conway, Markushevich, Titchmarsh}
cover Dirichlet series or the Riemann zeta function to a significant extent, although they do not prove the prime number
theorem.  Bak and Newman \cite[Sec.~19.5]{Bak} does an admirable job, although their presentation is dense (five pages).
Marshall's new book assigns the proof as a multi-part exercise that occupies half a page \cite[p.~191]{Marshall}.
Simon's four-volume treatise on analysis \cite{Simon} and the Stein--Shakarchi analysis series \cite{Shakarchi} devote a considerable
amount of space to topics in analytic number theory and include proofs of the prime number theorem.  
Lang's graduate-level complex analysis text \cite{Lang} thoroughly treats the prime number theorem, although he punts at a crucial point with an apparent note-to-self ``(Put the details as an exercise)''.   

On the other hand, number theory texts may present interesting digressions or tangential results that are not strictly
necessary for the proof of the prime number theorem.  They sometimes
suppress or hand wave through the complex analysis details we hope to exemplify.  
All of this may make navigating and outlining a streamlined proof difficult for the nonspecialist.
We do not give a guided tour of analytic number theory, nor do we dwell on results or notation that are unnecessary for our main goal: to present an efficient proof of the prime number theorem suitable for inclusion in a 
complex analysis course by an instructor who is not an expert in number theory.
For example, we avoid the introduction of general infinite products and Dirichlet series, Chebyshev's function $\psi$ and its integrated cousin $\psi_1$, the von Mangoldt function, the Gamma function, the Jacobi theta function, Poisson summation, and other staples of typical proofs.  Some fine number theory texts which contain 
complex-analytic proofs of the prime number theorem are
\cite{Apostol, Edwards, Fine, Ivic, Luca, Tenenbaum2}.

No instructor wants to be surprised in the middle of the lecture by a major logical gap in their notes.
Neither do they wish to assign problems that they later find are inaccurately stated or require
theorems that should have been covered earlier.
We hope that our presentation here will alleviate these difficulties.  
That is, we expect that a complex analysis instructor can use as much or as little of our
proof as they desire, according to the level of rigor and detail that they seek.
No step is extraneous and every detail is included.

The proof we present is based on Zagier's \cite{Zagier} presentation of Newman's proof \cite{Newman}
(see also Korevaar's exposition \cite{Korevaar}).
For our purposes their approach is ideal: it involves a minimal amount of number theory and a maximal amount of complex analysis.  
The number-theoretic content of our proof is almost trivial:
only the fundamental theorem of arithmetic and the definition of prime numbers are needed.
Although there are elementary proofs \cite{Erdos, Selberg}, in the sense that no complex analysis is required, these are obviously unsuitable for a complex analysis course.

This paper is organized as follows.  Each section is brief, providing the instructor with
bite-sized pieces that can be tackled in class or in (potentially inquiry-based) assignments.
We conclude many sections with related remarks that highlight common conceptual issues
or opportunities for streamlining if other tools, such as Lebesgue integration, are available.
Proofs of lemmas and theorems are often broken up into short steps
for easier digestion or adaptation as exercises.
Section \ref{Section:Asymptotic} introduces the prime number theorem and asymptotic equivalence $(\sim)$.
We introduce the Riemann zeta function $\zeta(s)$ in Section \ref{Section:RiemannZeta}, along with the Euler product formula.
In Section \ref{Section:Continue} we prove the zeta function 
has a meromorphic continuation to $\Re s > 0$.
We obtain series representations for $\log \zeta(s)$ and $\log|\zeta(s)|$ in Section \ref{Section:Log}.
These are used in Section \ref{Section:Nonvanish} 
to establish the nonvanishing of the zeta function on the vertical line $\Re s = 1$.
Section \ref{Section:Theta} introduces Chebyshev's function $\vartheta(x) = \sum_{p \leq x} \log p$ and establishes
a simple upper bound (needed later in Section \ref{Section:Newman}).  In 
Section \ref{Section:Phi}, we prove that a function related to $\log \zeta(s)$ extends analytically to
an open neighborhood of the closed half plane $\Re s \geq 1$.
Section \ref{Section:Laplace} provides a brief lemma on the analyticity of Laplace transforms.
Section \ref{Section:Newman} is devoted to the proof of Newman's Tauberian theorem, a true festival of complex analysis.
Section \ref{Section:Improper} uses Newman's theorem 
to establish the convergence of a certain improper integral, 
which is shown to imply $\vartheta(x) \sim x$ in Section \ref{Section:ThetaAsymptotic}.
We end in Section \ref{Section:Conclusion} with the conclusion of the proof of the prime number theorem.

\noindent\textbf{Acknowledgments.} We thank Ken Ribet, S.~Sundara Narasimhan, and Robert Sachs for helpful comments.

\section{Prime Number Theorem}\label{Section:Asymptotic}
Suppose that $f(x)$ and $g(x)$ are real-valued functions that are defined and nonzero for sufficiently large $x$.  
We write $f(x) \sim g(x)$ if
\begin{equation*}
\lim_{x\to\infty} \frac{f(x)}{g(x)} = 1
\end{equation*}
and we say that $f$ and $g$ are \emph{asymptotically equivalent} when this occurs.
The limit laws from calculus imply that $\sim$ is an equivalence relation.

Let $\pi(x)$ denote the number of primes at most $x$.
For example, $\pi(10.5) = 4$ since $2,3,5,7 \leq 10.5$.
The distribution of the primes appears somewhat erratic on the small scale.  For example,
we believe that there are infinitely many twin primes; that is, primes like $29$ and $31$
which differ by $2$ (this is the famed twin prime conjecture).  
On the other hand, there are arbitrarily large gaps between primes:
$n!+2, n!+3,\ldots,n!+n$ is a string of $n-1$ \emph{composite} (non-prime) numbers
since $n!+k$ is divisible by $k$
for $k=1,2,\ldots,n$.

The following landmark result is one of the
crowning achievements of human thought.  Although first conjectured by Legendre \cite{Legendre} around 1798 and perhaps a few years earlier by the young Gauss, 
it was proved independently by Hadamard \cite{Hadamard} and de la Vall\'ee Poussin in 1896 \cite{DVP}
with methods from complex analysis, building upon the seminal 1859 paper of Riemann \cite{RiemannWork}
(these historical papers are reprinted in the wonderful volume \cite{Borwein}).

\begin{theorem}[Prime Number Theorem]
\begin{equation*}
\pi(x) \sim \Li(x),
\end{equation*}
in which
\begin{equation*}
\Li(x) = \int_2^x \frac{dt}{\log t}
\end{equation*}
is the \emph{logarithmic integral}.
\end{theorem}

\begin{figure}
    \centering
    \begin{subfigure}[t]{0.45\textwidth}
        \centering
        \includegraphics[width=\textwidth]{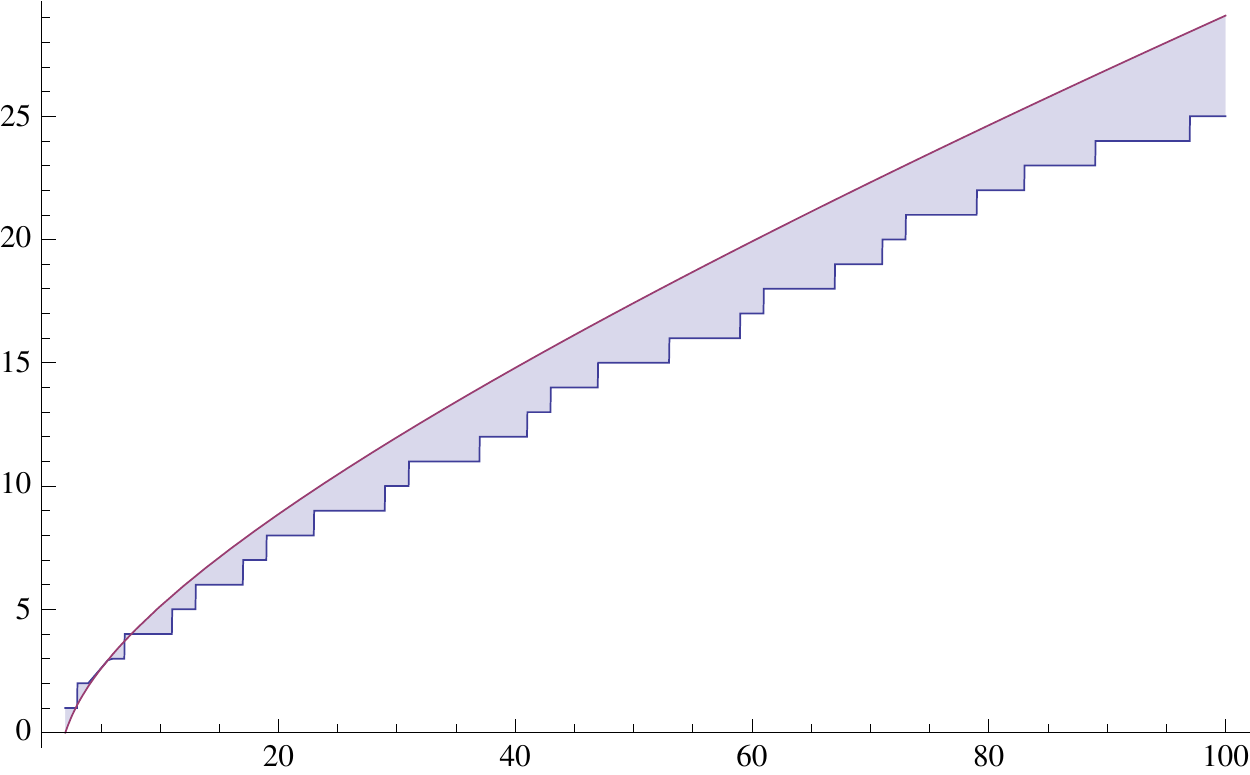}
        \caption{$x \leq 100$}
    \end{subfigure}\quad
    \begin{subfigure}[t]{0.45\textwidth}
        \centering
        \includegraphics[width=\textwidth]{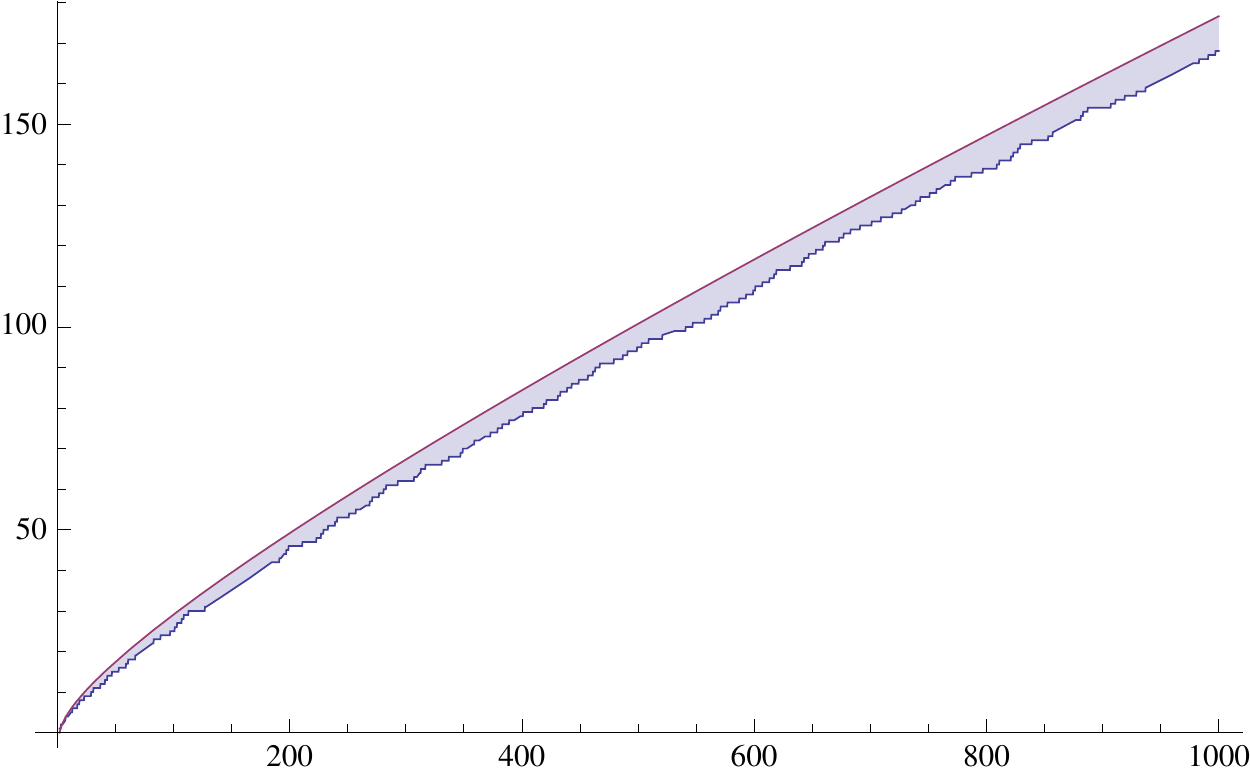}
        \caption{$x \leq 1{,}000$}
    \end{subfigure}
\\
    \begin{subfigure}[t]{0.45\textwidth}
        \centering
        \includegraphics[width=\textwidth]{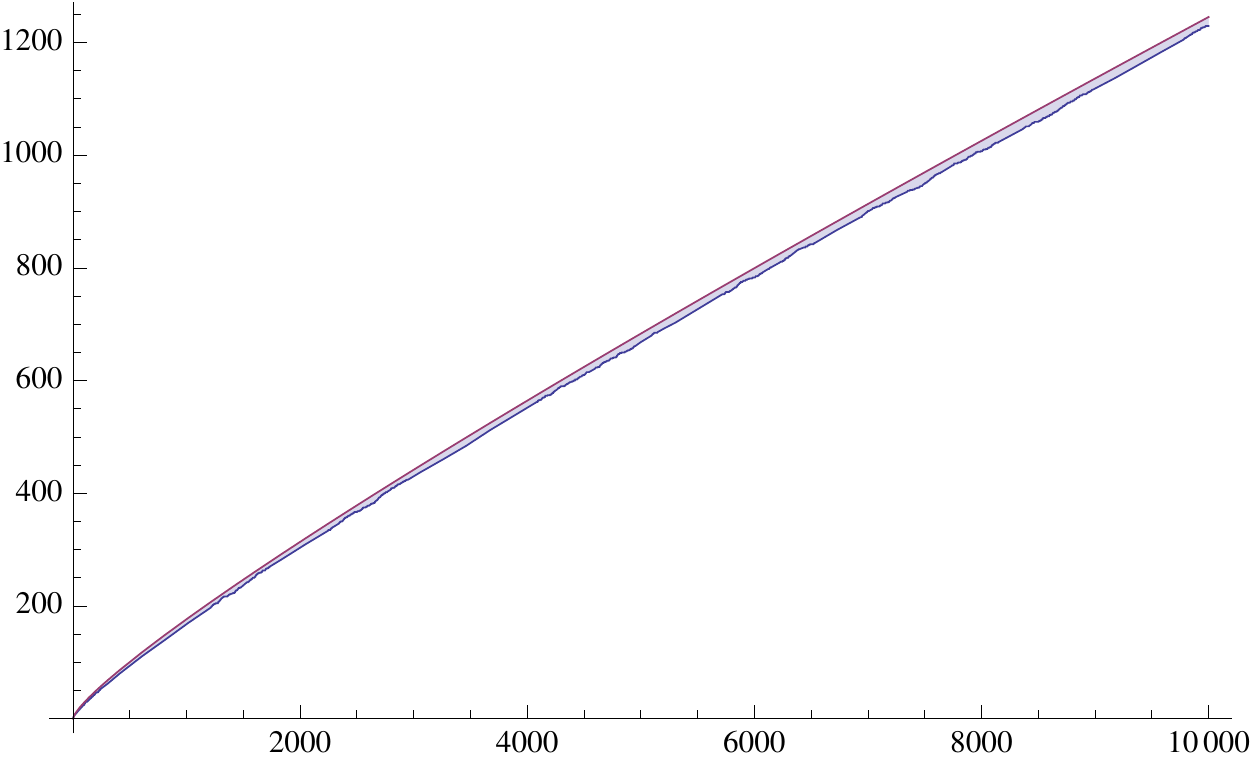}
        \caption{$x \leq 10{,}000$}
    \end{subfigure}\quad
    \begin{subfigure}[t]{0.45\textwidth}
        \centering
        \includegraphics[width=\textwidth]{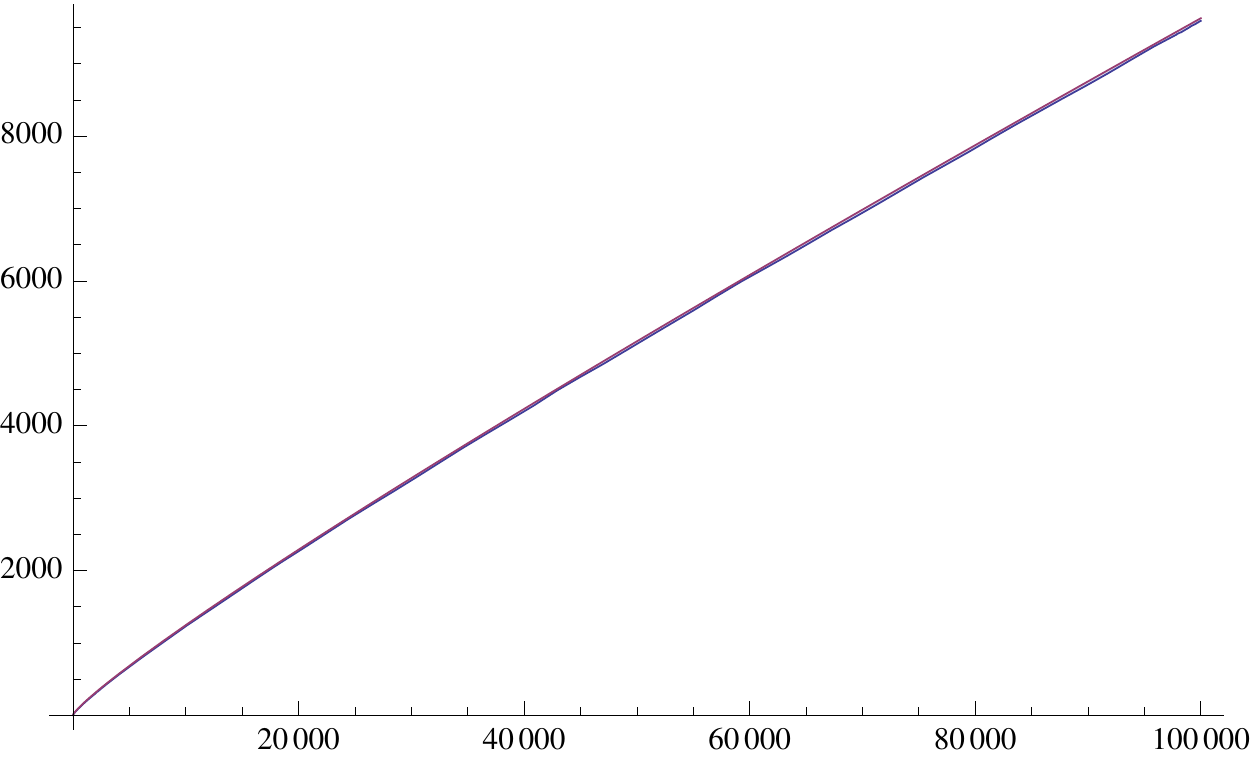}
        \caption{$x \leq 100{,}000$}
    \end{subfigure}
    \caption{Graphs of $\operatorname{Li}(x)$ versus $\pi(x)$ on various scales. }
    \label{Figure:PNT_Li}
\end{figure}

The predictions afforded by the prime number theorem are astounding; 
see Figure \ref{Figure:PNT_Li}.
Unfortunately, $\Li(x)$ cannot be evaluated in closed form.
As a consequence, it is convenient to replace $\Li(x)$ 
with a simpler function that is asymptotically equivalent to it.
L'H\^opital's rule and the fundamental theorem of calculus imply that
\begin{equation*}
\lim_{x\to\infty} \frac{  \Li(x) }{x/\log x}
\overset{L}{=} 
\lim_{x\to\infty} \frac{ \frac{1}{\log x}}{ \frac{\log x - x(\frac{1}{x})}{(\log x)^2}}
= \lim_{x\to\infty}  \frac{1}{1 - \frac{1}{\log x}} 
=1
\end{equation*}
and hence
\begin{equation*}
\Li(x) \sim \frac{x}{\log x}.
\end{equation*}
However, the logarithmic integral provides a better approximation to $\pi(x)$;
see Table \ref{Table:LogInt}.  

\begin{table}\small
    \begin{tabular}{c|ccc}
    $x$ & $\pi(x)$ & $\operatorname{Li}(x)$ & $x / \log x$ \\
    \hline
    1000 & 168 & 177 & 145 \\
    10{,}000 & 1{,}229 &  1{,}245 & 1{,}086 \\
    100{,}000 &  9{,}592 & 9{,}629 & 8{,}686 \\
    1{,}000{,}000 & 78{,}498 & 78{,}627 & 72{,}382\\
    10{,}000{,}000 & 664{,}579 & 664{,}917 & 620{,}421 \\
    100{,}000{,}000 & 5{,}761{,}455 & 5{,}762{,}208 & 5{,}428{,}681 \\
    1{,}000{,}000{,}000 & 50{,}847{,}534 & 50{,}849{,}234 & 48{,}254{,}942\\
    10{,}000{,}000{,}000 & 455{,}052{,}511 & 455{,}055{,}614 & 434{,}294{,}482 \\
    100{,}000{,}000{,}000 & 4{,}118{,}054{,}813 & 4{,}118{,}066{,}400 &  3{,}948{,}131{,}654 \\
    1{,}000{,}000{,}000{,}000 & 37{,}607{,}912{,}018 & 37{,}607{,}950{,}280 & 36{,}191{,}206{,}825 \\
    \end{tabular}
    \smallskip
    \caption{The logarithmic integral $\operatorname{Li}(x)$ is a better approximation to the prime counting function
    $\pi(x)$ than is $x / \log x$ (entries rounded to the nearest integer).}
    \label{Table:LogInt}
\end{table}

We will prove the prime number theorem in the following equivalent form.

\begin{theorem}[Prime Number Theorem]\label{Theorem:PNT}
$\displaystyle \pi(x) \sim \frac{x}{\log x}$.
\end{theorem}

Our proof incorporates modern simplifications due to Newman \cite{Newman} and Zagier \cite{Zagier}.  
However, the proof is still difficult and involves most of the techniques and tools from a
typical complex analysis course.  There is little number theory in the proof; it is almost all complex analysis.
Consequently, it is an eminently
fitting capstone for a complex analysis course.
As G.H.~Hardy opined in 1921 \cite{Littlewood}:
\begin{quote}\small
No elementary proof of the prime number theorem is known, and one may ask whether it is reasonable to expect one. Now we know that the theorem is roughly equivalent to a theorem about an analytic function, the theorem that Riemann's zeta function has no roots on a certain line. A proof of such a theorem, not fundamentally dependent on the theory of functions [complex analysis], seems to me extraordinarily unlikely.
\end{quote}

In 1948  Erd\H{o}s \cite{Erdos} and Selberg \cite{Selberg} independently found proofs of the prime number theorem that avoid complex analysis.  
These ``elementary'' proofs are more difficult and intricate than the approach presented here; see \cite{Diamond, Tenenbaum, HardyWright, Levinson, Edwards} for the details
and \cite{Goldfeld, Spencer} for an account of the murky history of the elementary proof.

\begin{remark}
A common misconception is that $f(x) \sim g(x)$ implies that 
$f(x) - g(x)$ tends to zero, or that it remains small.  The functions
$f(x) = x^2 + x$ and $g(x) = x^2$ are asymptotically equivalent, yet
their difference is unbounded.
\end{remark}

\begin{remark}
The prime number theorem implies that $p_n \sim n \log n$, in which $p_n$ denotes the $n$th prime number.
Since $\pi(p_n) = n$, substitute $q = p_n$ and obtain
\begin{align*}
\lim_{n\to\infty} \frac{n \log n}{p_n}
&=  \lim_{n\to\infty}  \left( \frac{\pi(p_n) \log p_n}{p_n} \right) \left( \frac{\log n}{\log p_n} \right)
=  \lim_{n\to\infty}  \frac{\log n}{\log p_n} \\
&=  \lim_{q\to\infty}  \frac{\log \pi(q)}{\log q} 
=  \lim_{q\to\infty}  \frac{\log \left( \frac{\pi(q)\log q}{q} \right) + \log q - \log\log q}{\log q} \\
&=  \lim_{q\to\infty} \left( \frac{\log 1}{\log q} + 1 - \frac{\log \log q}{\log q} \right) 
=  1. 
\end{align*}
\end{remark}

\begin{remark}
Another simple consequence of the prime number theorem is the density of 
$\{ p/q : \text{$p,q$ prime}\}$ in $[0,\infty)$ \cite{QSDE}.
\end{remark}

\section{The Riemann zeta function}\label{Section:RiemannZeta}
    The \emph{Riemann zeta function} is defined by 
    \begin{equation}\label{eq:Zeta}\qquad
        \zeta(s) = \sum_{n=1}^{\infty} \frac{1}{n^s},\quad \text{for $\Re s > 1$}.
    \end{equation}
    The use of $s$ for a complex variable is standard in analytic number theory, and we largely adhere to this convention.    
    Suppose that $\Re s \geq \sigma > 1$.  Since 
    \begin{equation*}
        |n^s| = |e^{s \log n}| = e^{\Re (s \log n)} = e^{(\log n)\Re s} 
        = (e^{\log n})^{\Re s} = n^{\Re s} \geq n^{\sigma}
    \end{equation*}
    it follows that
    \begin{equation*}
        \left| \sum_{n=1}^{\infty} \frac{1}{n^s} \right| \leq \sum_{n=1}^{\infty} \frac{1}{n^{\sigma}} < \infty.
    \end{equation*}
    The Weierstrass $M$-test ensures that \eqref{eq:Zeta} converges 
    absolutely and uniformly on $\Re s \geq \sigma$.  
    Since $\sigma > 1$ is arbitrary and each summand in \eqref{eq:Zeta} is analytic on $\Re s > 1$, we conclude that \eqref{eq:Zeta} converges
    locally uniformly on $\Re s > 1$ to an analytic function.

    In what follows, $p$ denotes a prime number and a sum or product indexed by $p$
    runs over the prime numbers.
    Here is the connection between the zeta function and the primes.
    
    \begin{theorem}[Euler Product Formula]\label{Theorem:EulerProduct}
        If $\Re s > 1$, then $\zeta(s) \neq 0$ and 
        \begin{equation}\label{eq:EulerProduct}
        	\zeta(s) =\
        	\prod_{p} \bigg( 1 - \frac{1}{p^s} \bigg)^{-1}.
        \end{equation}
        The convergence is locally uniform in $\Re s > 1$.
    \end{theorem}

\begin{proof}
Since $|p^{-s}| = p^{-\Re s} < 1$ for $\Re s > 1$, the geometric series formula implies
\begin{equation*}
\bigg( 1 - \frac{1}{p^s} \bigg)^{-1} = \sum_{n=0}^{\infty} \bigg(\frac{1}{p^s}\bigg)^n = \sum_{n=0}^{\infty} \frac{1}{p^{ns}},
\end{equation*}
in which the convergence is absolute.  Since a finite number of absolutely convergent series can be multiplied term-by-term, it follows that
\begin{align*}
\bigg( 1 - \frac{1}{2^s} \bigg)^{-1}\bigg( 1 - \frac{1}{3^s} \bigg)^{-1}
&=\left(1 + \frac{1}{2^s} + \frac{1}{2^{2s}} + \cdots \right)\left(1 + \frac{1}{3^s} + \frac{1}{3^{2s}} + \cdots \right) \\
&=1 + \frac{1}{2^s} + \frac{1}{3^s} + \frac{1}{4^s}+ \frac{1}{6^s} + \frac{1}{8^s} + \frac{1}{9^s} + \frac{1}{12^s} + \cdots,
\end{align*}
in which only natural numbers divisible by the primes $2$ or $3$ appear.
Similarly,
\begin{align*}
\prod_{p\leq 5} \bigg( 1 - \frac{1}{p^s} \bigg)^{-1}
&=\left(1 + \frac{1}{2^s} + \frac{1}{3^s} + \frac{1}{4^s}+ \frac{1}{6^s} + \frac{1}{8^s}  + \cdots\right)\left(1 + \frac{1}{5^s}+ \frac{1}{5^{2s}}+\cdots\right) \\
&=1 + \frac{1}{2^s} + \frac{1}{3^s} + \frac{1}{4^s}+ \frac{1}{5^s}+ \frac{1}{6^s} + \frac{1}{8^s} + \frac{1}{9^s} + \frac{1}{10^s} + \frac{1}{12^s} + \cdots,
\end{align*}
in which only natural numbers divisible by the primes $2$, $3$, or $5$ appear.
Since the prime factors of each $n \leq N$ are at most $N$,
and because the tail of a convergent series tends to $0$, it follows that
for $\Re s \geq \sigma > 1$
\begin{equation*}
\left| \zeta(s) - \prod_{p\leq N} \bigg( 1 - \frac{1}{p^s} \bigg)^{-1}\right|
\leq \sum_{n > N} \left| \frac{1}{n^s} \right| 
\leq  \sum_{n=N}^{\infty} \frac{1}{n^{\sigma}} \to 0
\end{equation*}
as $N \to\infty$.  This establishes \eqref{eq:EulerProduct} and proves
that the convergence is locally uniform on $\Re s > 1$.  Since each partial product does not vanish on $\Re s >1$
and because the limit $\zeta(s)$ is not identically zero, 
Hurwitz' theorem\footnote{Let $\Omega \subseteq \C$ be nonempty, connected, and open, and let $f_n$ be a sequence of analytic functions that converges locally uniformly on $\Omega$ to $f$ (which is necessarily analytic). 
If each $f_n$ is nonvanishing on $\Omega$, then $f$ is either identically zero or nowhere vanishing on $\Omega$.}
ensures that $\zeta(s) \neq 0$ for $\Re s > 1$.
\end{proof}

\begin{remark}
By $\prod_{p} ( 1 - p^{-s})^{-1}$ we mean $\lim_{N\to\infty} \prod_{p\leq N}( 1 - p^{-s})^{-1}$.
This definition is sufficient for our purposes, but differs from the general definition of infinite products
(in terms of logarithms)
one might see in advanced complex-variables texts.
\end{remark}

\begin{remark}
The convergence of $\prod_{p}( 1 - p^{-s})^{-1}$ and the nonvanishing of each factor 
does not automatically imply that the infinite product is nonvanishing
(this is frequently glossed over).
Indeed, $\lim_{N\to\infty} \prod_{n=1}^N \frac{1}{2} = \frac{1}{2^N} = 0$
even though each factor is nonzero.  Thus, the appeal to Hurwitz' theorem is necessary
unless another approach is taken.
\end{remark}

\begin{remark}
A similar argument establishes
\begin{equation}\label{eq:SDHG}
\zeta(s) \prod_p \left( 1 - \frac{1}{p^s} \right) = 1,
\end{equation}
in which the convergence is locally uniform on $\Re s > 1$.  This directly yields 
the nonvanishing of $\zeta(s)$ on $\Re s > 1$.  
However, a separate argument is needed to deduce the locally uniform convergence of \eqref{eq:EulerProduct}
from the locally uniform convergence of \eqref{eq:SDHG}.
\end{remark}

\begin{remark}
The Euler product formula implies Euclid's theorem (the infinitude of the primes).
If there were only finitely many primes, then the right-hand side of \eqref{eq:EulerProduct} would
converge to a finite limit as $s \to 1^+$.  However, the left-hand side of \eqref{eq:Zeta}
diverges as $s \to 1^+$ since its terms tend to those of the harmonic series.
\end{remark}

\section{Analytic Continuation of the Zeta Function}\label{Section:Continue}
We now prove that the Riemann zeta function can be analytically continued to $\Re s >0$, with the exception
of $s=1$, where $\zeta(s)$ has a simple pole.  Although much more can be said about this matter, this modest
result is sufficient for our purposes.  On the other hand, the instructor might wish to supplement this 
material with some remarks on the Riemann Hypothesis; see Remark \ref{Remark:RH}.  Students perk up at the mention of the
large monetary prize associated to the problem.
At the very least, they may wish to learn about the most important
open problem in mathematics.

\begin{theorem}\label{Theorem:ZetaExtend}
    $\zeta(s)$ can be analytically continued to $\Re s > 0$ except for
    a simple pole at $s=1$ with residue $1$.
\end{theorem}

\begin{proof}
In what follows, $\floor{x}$ denotes the unique
integer such that $\floor{x} \leq x < \floor{x}+1$; in particular, $0 \leq x - \floor{x} < 1$.
For $\Re s > 2$,\footnote{The assumption $\Re s > 2$ ensures that both
$\sum_{n=1}^{\infty} \frac{n}{n^s}$ and $\sum_{n=1}^{\infty} \frac{n-1}{n^s}$ converge locally uniformly.}%
\begin{align*}
\zeta(s)
&= \sum_{n=1}^{\infty} \frac{1}{n^s} 
= \sum_{n=1}^{\infty} \frac{n-(n-1)}{n^s} \\
&= \sum_{n=1}^{\infty} \frac{n}{n^s} - \sum_{n=2}^{\infty} \frac{n-1}{n^s}
= \sum_{n=1}^{\infty} \frac{n}{n^s} - \sum_{n=1}^{\infty} \frac{n}{(n+1)^s}\\
&= \sum_{n=1}^{\infty} n \left( \frac{1}{n^s} - \frac{1}{(n+1)^s} \right)\\
&= \sum_{n=1}^{\infty} n \left(s \int_n^{n+1} \frac{dx}{x^{s+1}} \right)\\
&= s \sum_{n=1}^{\infty}  \int_n^{n+1} \frac{n  \,dx}{x^{s+1}} \\
&= s \sum_{n=1}^{\infty}  \int_n^{n+1} \frac{\floor{x}\,  dx}{x^{s+1}} \\
&= s \int_1^{\infty} \frac{\floor{x} \, dx}{x^{s+1}}.
\end{align*}
Observe that for $\Re s > 1$,
\begin{equation*}
\int_1^{\infty} \frac{dx}{x^{s}} = \frac{1}{s-1} \quad\implies\quad
\frac{1}{s-1} + 1 - s \int_1^{\infty} \frac{x}{x^{s+1}}\,dx = 0
\end{equation*}
and hence 
\begin{align}
\zeta(s) 
&= s \int_1^{\infty} \frac{\floor{x} \, dx}{x^{s+1}} \nonumber \\
&= \left(\frac{1}{s-1} + 1 - s \int_1^{\infty} \frac{x}{x^{s+1}}\,dx \right) + s \int_1^{\infty} \frac{\floor{x} \, dx}{x^{s+1}} \nonumber \\
&= \frac{1}{s-1} + 1 - s \int_1^{\infty} \frac{x-\floor{x}}{x^{s+1}}\,dx.\label{eq:ThisIntegral}
\end{align}
If the integral above defines an analytic function on $\Re s > 0$,
then $\zeta(s)$ can be analytically continued to $\Re s > 0$ except for
a simple pole at $s=1$ with residue $1$.   We prove this with techniques commonly available
at the undergraduate-level (see Remark \ref{Remark:Lebesgue}).

For $n =1,2,\ldots$, let
\begin{equation*}
f_n(s) = \int_n^{n+1} \frac{x-\floor{x}}{x^{s+1}}\,dx .
\end{equation*}
For any simple closed curve $\gamma$ in $\Re s > 0$, 
Fubini's theorem and Cauchy's theorem imply
\begin{align*}
\int_{\gamma} f_n(s)\,ds
&=\int_{\gamma} \int_n^{n+1} \frac{x-\floor{x}}{x^{s+1}}\,dx \,ds\\
&= \int_n^{n+1}\!\!\! (x-\floor{x}) \left(\int_{\gamma}\frac{ds}{x^{s+1}}\right) \,dx\\
&= \int_n^{n+1}\!\!\! (x-\floor{x}) \, 0 \,dx \\
&=0.
\end{align*}
Morera's theorem ensures that each $f_n$ is analytic on $\Re s > 0$.  
If $\Re s \geq \sigma > 0$, then
\begin{align*}
\sum_{n=1}^{\infty} |f_n(s)|
&= \sum_{n=1}^{\infty} \left| \int_n^{n+1} \frac{x-\floor{x}}{x^{s+1}}\,dx\right| \\
&\leq \sum_{n=1}^{\infty}  \int_n^{n+1} \left|\frac{x-\floor{x}}{x^{s+1}}\right|\,dx \\
&\leq \sum_{n=1}^{\infty}  \int_n^{n+1} \frac{dx}{x^{\Re(s+1)}}\\
&\leq \int_1^{\infty}\frac{dx}{x^{\sigma+1}} \\
&= \frac{1}{\sigma} \\
& < \infty.
\end{align*}
Consequently, the Weierstrass $M$-test implies that
\begin{equation}\label{eq:Floor}
\sum_{n=1}^{\infty} f_n(s)  = \int_1^{\infty} \frac{x-\floor{x}}{x^{s+1}}\,dx
\end{equation}
converges absolutely and uniformly on $\Re s \geq \sigma$.  Since $\sigma > 0$
was arbitrary, it follows that the series converges locally uniformly on $\Re s > 0$.
Being the locally uniform limit of analytic functions on $\Re s >0$, we conclude that 
\eqref{eq:Floor} is analytic there.
\end{proof}

\begin{remark}\label{Remark:Lebesgue}
The instructor should be aware that many sources, in the interest of brevity, 
claim without proof that the integral in \eqref{eq:ThisIntegral} defines an analytic
function on $\Re s > 0$.  This is a nontrivial result for an undergraduate course, especially since the
domain of integration is infinite.  If the instructor has Lebesgue integration at their disposal, 
the dominated convergence theorem, which can be applied to $[1,\infty)$, makes the proof significantly easier.
\end{remark}

\begin{remark}\label{Remark:RH}
It turns out that $\zeta(s)$ can be analytically continued to $\C \backslash\{1\}$.
The argument involves the introduction of the gamma function $\Gamma(z) = \int_0^{\infty} x^{z-1} e^{-x}\,dx$ to obtain
the \emph{functional equation}
\begin{equation}\label{eq:Functional}
\zeta (s)=2^{s}\pi ^{s-1} \sin \left({\frac {\pi s}{2}}\right) \Gamma (1-s) \zeta (1-s).
\end{equation}
The extended zeta function has zeros at $-2,-4,-6,\ldots$ (the \emph{trivial zeros}),
along with infinitely many zeros in the \emph{critical strip} $0 < \Re s < 1$ (the \emph{nontrivial zeros}).  
To a few decimal places, here are the first twenty nontrivial zeros that lie in the upper half plane (the zeros are symmetric with respect to the real axis):
\begin{align*}
0.5 +14.1347 i,\,\,
0.5 +21.0220 i,\,\,
0.5 +25.0109 i,\,\,
0.5 +30.4249 i,\,\,
0.5 +32.9351 i,\\
0.5 +37.5862 i,\,\,
0.5 +40.9187 i,\,\,
0.5 +43.3271 i,\,\,
0.5 +48.0052 i,\,\,
0.5 +49.7738 i,\\
0.5 + 52.9703 i,\,\,
0.5 + 56.4462 i,\,\,
0.5 + 59.3470 i,\,\,
0.5 + 60.8318 i,\,\,
0.5 + 65.1125 i,\\
0.5 + 67.0798 i,\,\,
0.5 + 69.5464 i,\,\,
0.5 + 72.0672 i,\,\,
0.5 + 75.7047 i, \,\,
0.5 + 77.1448 i.
\end{align*}
The first $10^{13}$ nontrivial
zeros lie on the \emph{critical line} $\Re s = \frac{1}{2}$.
The famous \emph{Riemann Hypothesis}
asserts that all the zeros in the critical strip lie on the critical line; see Figure \ref{Figure:CriticalStrip}.  This problem 
was first posed by Riemann in 1859 and remains unsolved.  It is considered the most important open problem in mathematics
because of the impact it would have on the distribution of the prime numbers.  

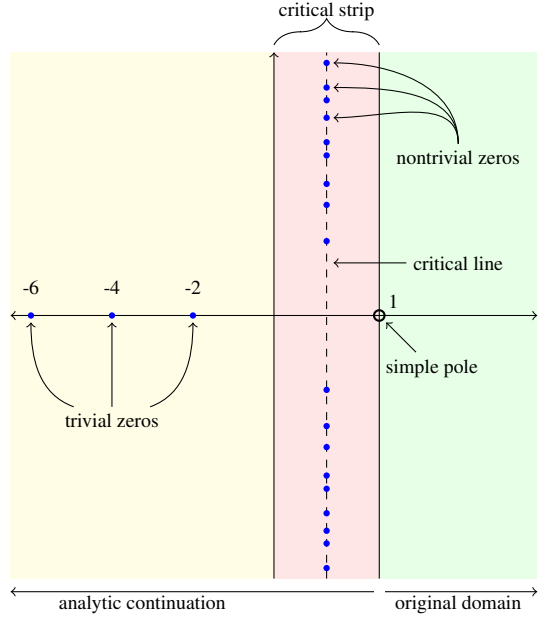
\begin{figure}
    \begin{tikzpicture}[scale=0.7, every node/.style={scale=0.75}]
       \fill[red,opacity=.1] (0,-5) rectangle (2,5);
       \fill[green,opacity=.1] (2,-5) rectangle (5,5);
       \fill[yellow,opacity=.1] (-5,-5) rectangle (0,5);
       \draw[<->,thin](-5,0)--(5,0)node[above]{};
       \draw[<->,thin](-5,-5.25)--(5,-5.25);
       \draw[->,thin](0,-5)--(0,5)node[left]{};
       \draw[](2,-5)--(2,5);
       \draw[dashed](1,-5)--(1,5);

       \foreach \x in {-2,-4,-6}
            \fill[blue] (\x/1.3,0) circle (0.06cm)node[yshift = .5cm]{\color{black}\x};

        \foreach \y in {14.1347,21.0220,25.0109,30.4249,32.9351,37.5862,40.9187,43.3271,48.0052}{
            \fill[blue] (1,\y/10) circle (0.06cm);
            \fill[blue] (1,-\y/10) circle (0.06cm);
        }

        \draw[thick] (2,0) circle (.1cm)node[xshift=.25cm,yshift=.25cm]{1};

        \node (a) at (2,0){};
        \node (b) at (3,-1){simple pole};
        \draw[<-] (a) -- (b);

        \node(c) at (1,1){};
        \node(d) at (3.5,1){critical line};
        \draw[<-] (c) -- (d);

        \draw [decorate,decoration={brace,amplitude=10pt},yshift=4pt]
(0,5) -- (2,5) node [black,midway,yshift=0.6cm]
{critical strip};

        \node (e) at (-2/1.3,0){};
        \node (f) at (-4/1.3,0){};
        \node (g) at (-6/1.3,0){};

        \node(h) at (-4/1.3,-2){trivial zeros};

        \draw[->] (h) to [in=270,out=20](e);
        \draw[->] (h) to [in=270,out=160](g);
        \draw[->] (h) to (f);

        \node (i) at (3.5,3){nontrivial zeros};

        \node (j) at (1,43.3271/10){};
        \node (k) at (1,48.0052/10){};
        \node (l) at (1,37.5862/10){};

        \draw[->] (i) to [in=0,out = 90](j);
        \draw[->] (i) to [in=0,out = 90](k);
        \draw[->] (i) to [in=0,out = 90](l);

        \fill[white](1.9,-5.3) rectangle (2.1,-5.2);

        \node at (3.5,-5.5) {original domain};
        \node at (-2.5,-5.5) {analytic continuation};
    \end{tikzpicture}
    \caption{Analytic continuation of $\zeta(s)$ to $\C \backslash \{1\}$.
    The nontrivial zeros of the Riemann zeta function lie in the \emph{critical strip}
    $0 < \Re s < 1$.  The Riemann Hypothesis asserts that all of them lie on the \emph{critical line} $\Re s= \frac{1}{2}$.}    
    \label{Figure:CriticalStrip}
\end{figure}
\end{remark}

\begin{remark}
Students might benefit from learning that the error in the estimate afforded by the prime number theorem
is tied to the zeros of the zeta function.  Otherwise the Riemann Hypothesis might seem too esoteric and
unrelated to the prime number theorem.
One can show that if $\zeta(s) \neq 0$ for $\Re s > \sigma$, then 
there is a constant $C_{\sigma}$ such that 
\begin{equation*}
| \pi(x) - \Li(x) | \leq C_{\sigma} x^{\sigma}\log x
\end{equation*}
for all $x \geq 2$ \cite[(2.2.6)]{Borwein}.  
Since it is known that the zeta function has infinitely many zeros on the critical line $\Re s = \frac{1}{2}$, 
we must have $\sigma \geq \frac{1}{2}$.  
\end{remark}

\section{The logarithm of $\zeta(s)$}\label{Section:Log}
    In this section we establish a series representation of the logarithm of the zeta function.
    We use this in Section \ref{Section:Nonvanish} to establish the nonvanishing
    of $\zeta(s)$ for $\Re s = 1$ and in Section \ref{Section:Phi} to obtain an analytic continuation of a closely-related function.

    \begin{lemma}\label{Lemma:ZetaLogC}
        If $\Re s > 1$, then 
        $\displaystyle \log \zeta(s) =\sum_{n=1}^{\infty} \frac{c_n }{n^s}$, in which $c_n  \geq 0$ for $n \geq 1$.
    \end{lemma}

    \begin{proof}
        The open half plane $\Re s > 1$ is simply connected and $\zeta(s)$ does not vanish there 
        (Theorem \ref{Theorem:EulerProduct}).
        Thus, we may define a branch of $\log \zeta(s)$ for $\Re s >1$ such that $\log \zeta(\sigma) \in \R$
        for $\sigma > 1$.      
        Recall that
        \begin{equation}\label{eq:Prev}
            \log \left( \frac{1}{1-z} \right) = \sum_{k=1}^{\infty} \frac{z^k}{k}
        \end{equation}
        for $|z| < 1$ and observe that $\Re s > 1$ implies $|p^{-s}| = p^{-\Re s} < 1$, which permits $z = p^{-s}$ in \eqref{eq:Prev}.
        The Euler product formula \eqref{eq:EulerProduct}, the nonvanishing of $\zeta(s)$ for $\Re s > 1$, and the continuity of the logarithm imply
        \begin{align}
            \log \zeta(s)
            &= \log \prod_p \left( \frac{1}{1 - p^{-s}} \right)
            = \sum_p \log \left(\frac{1}{1 - p^{-s}} \right)  \label{eq:Thing} \\
            &= \sum_p \sum_{k=1}^{\infty} \frac{(p^{-s})^k}{k} 
            = \sum_p \sum_{k=1}^{\infty} \frac{1/k}{ (p^{k})^s} 
            = \sum_{n=1}^{\infty} \frac{c_n }{n^s},\nonumber
        \end{align}
        in which
        \begin{equation}\label{eq:Mangle}
            c_n  = 
            \begin{cases}
                \dfrac{1}{k} & \text{if $n = p^k$},\\[5pt]
                0 & \text{otherwise}.
            \end{cases}
        \end{equation}
        The rearrangement of the series above is permissible by absolute convergence.
    \end{proof}

    \begin{lemma}\label{Lemma:ZetaAbs}
        If $s= \sigma + it$, in which $\sigma > 1$ and $t \in \R$, then 
        \begin{equation*}
            \log|\zeta(s)| = \sum_{n=1}^{\infty} \frac{c_n  \cos(t\log n)}{n^{\sigma}},
        \end{equation*}
        in which the $c_n$ are given by \eqref{eq:Mangle}.
    \end{lemma}
    
    \begin{proof}
        Since $\sigma = \Re s > 1$, Lemma \ref{Lemma:ZetaLogC} and Euler's formula provide
        \begin{align*}
            \log|\zeta(s)|
            &= \Re \big(\log \zeta(s) \big)
            = \Re \sum_{n=1}^{\infty} \frac{c_n }{n^{\sigma+it}}\\
            &= \Re \sum_{n=1}^{\infty} \frac{c_n}{e^{(\sigma+it)\log n}}
            = \Re \sum_{n=1}^{\infty} \frac{c_n e^{-it\log n }}{e^{\sigma\log n} }\\
            &=  \sum_{n=1}^{\infty} \frac{c_n \cos(t \log n) }{n^{\sigma}}. \qedhere
        \end{align*}
    \end{proof}

    \begin{remark}\label{Remark:EulerEuclid}
        The identity \eqref{eq:Thing} permits a proof that
        $\sum_p p^{-1}$ diverges; this is
        Euler's refinement of Euclid's theorem (the infinitude of the primes).
        Suppose toward a contradiction that $\sum_p p^{-1}$ converges.
        For $|z| < \frac{1}{2}$, \eqref{eq:Prev} implies
        \begin{equation}\label{eq:TheDarnInequality}
           \left| \log \left( \frac{1}{1-z} \right)\right| = \left| \sum_{k=1}^{\infty} \frac{z^k}{k} \right| \leq \sum_{k=1}^{\infty} |z|^k = \frac{|z|}{1-|z| } \leq 2|z|.
        \end{equation}
        For $s > 1$, \eqref{eq:Thing} and the previous inequality imply
        \begin{equation*}
            \log \zeta(s) = \sum_p \log \left(\frac{1}{1 - p^{-s}} \right) < 2\sum_p \frac{1}{p^s} \leq 2 \sum_p \frac{1}{p} < \infty.
        \end{equation*}
        This contradicts the fact that $\zeta(s)$ has a pole at $s=1$.
    The divergence of $\sum_p p^{-1}$ tells us that
    the primes are packed tighter in the natural numbers than are the perfect squares since
    $\sum_{n=1}^{\infty} \frac{1}{n^2} = \zeta(2)$ is finite (in fact, Euler proved that it equals $\frac{\pi^2}{6}$).
\end{remark}

\section{Nonvanishing of $\zeta(s)$ on $\Re s = 1$}\label{Section:Nonvanish}
Theorem \ref{Theorem:ZetaExtend} provides the analytic continuation of $\zeta(s)$ to $\Re s > 0$.
The following important result tells us that the extended zeta function does not vanish on the vertical line $\Re s = 1$.
One can show that this statement is equivalent to the prime number theorem,  although we focus only on deriving the prime number theorem from it.

\begin{theorem}\label{Theorem:ZetaZeros}
$\zeta(s)$ has no zeros on $\Re s = 1$.
\end{theorem}

\begin{proof}
Recall that $\zeta(s)$ extends analytically to $\Re s > 0$ (Theorem \ref{Theorem:ZetaExtend})
except for a simple pole at $s=1$; in particular, $\zeta(s)$ does not vanish at $s=1$.
Suppose toward a contradiction that $\zeta(1+it) = 0$ for some $t \in \R \backslash\{0\}$ and consider
\begin{equation*}
f(s)= \zeta^3(s) \zeta^4(s+it) \zeta(s+2it).
\end{equation*}
Observe that
\begin{enumerate}
\item $\zeta^3(s)$ has a pole of order three at $s=1$ since $\zeta(s)$ has a simple pole at $s=1$;
\item $\zeta^4(s+it)$ has a zero of order at least four at $s=1$ since $\zeta(1+it)=0$; and
\item $\zeta(s+2it)$ does not have a pole at $s=1$ since $t \in \R \backslash\{0\}$ and $s=1$ is the only pole of $\zeta(s)$ on $\Re s = 1$.
\end{enumerate}
Thus, the singularity of $f$ at $s=1$ is removable and $f(1) = 0$.  Therefore,
\begin{equation}\label{eq:ZetaLog}
\lim_{s\to 1} \log |f(s)| = - \infty.
\end{equation}
On the other hand, Lemma \ref{Lemma:ZetaAbs} yields
\begin{align*}
\log |f(s)|
&= 3 \log |\zeta(s)| + 4 \log |\zeta(s+it)| + \log |\zeta(s+2it)| \\
&=3  \sum_{n=1}^{\infty} \frac{c_n }{n^{\sigma}} + 4\sum_{n=1}^{\infty} \frac{c_n \cos (t \log n)}{n^{\sigma}}  + \sum_{n=1}^{\infty} \frac{c_n \cos (2t \log n)}{n^{\sigma}}  \\
&= \sum_{n=1}^{\infty} \frac{c_n }{n^{\sigma}} \big(3 + 4 \cos (t \log n) + \cos(2t \log n) \big) \\
&\geq 0
\end{align*}
since $c_n  \geq 0$ for $n\geq 1$ and 
\begin{equation*}
3 + 4 \cos x + \cos 2x = 2(1+ \cos x)^2 \geq 0, \qquad \text{for $x \in \R$}.
\end{equation*}
Since this contradicts \eqref{eq:ZetaLog}, we conclude that $\zeta(s)$ has no zeros with $\Re s= 1$.
\end{proof}

\begin{remark}
Since Theorem \ref{Theorem:EulerProduct} already ensures that $\zeta(s) \neq 0$ for $\Re s > 1$, Theorem \ref{Theorem:ZetaZeros} implies
$\zeta(s)$ does not vanish in the closed half plane $\Re s \geq 1$.  
\end{remark}

\section{Chebyshev Theta Function}\label{Section:Theta}
    It is often convenient to attack problems related to prime numbers with logarithmically
    weighted sums.  Instead of working with $\pi(x) = \sum_{p\leq x} 1$ directly, we consider
    \begin{equation}\label{eq:ThetaDef}
    \vartheta(x) = \sum_{p\leq x} \log p.
    \end{equation}
    We will derive the prime number theorem from the statement 
    $\vartheta(x) \sim x$.   Since this asymptotic equivalence is difficult to establish,
    we first content ourselves with an upper bound.
    
    \begin{theorem}[Chebyshev's Lemma]\label{Theorem:Chebyshev}
        $\vartheta(x) \leq 3 x$.
    \end{theorem}

\begin{proof}
If $n < p \leq 2n$, then
\begin{equation*}
p \quad \text{divides}\quad \binom{2n}{n} = \frac{(2n)!}{n!n!}
\end{equation*}
since $p$ divides the numerator but not the denominator.  The binomial theorem implies
\begin{align*}
2^{2n} &= (1+1)^{2n} = \sum_{k=0}^{2n} \binom{2n}{k}1^k1^{2n-k} \\
&\geq \binom{2n}{n} \geq \prod_{n < p \leq 2n}p = \prod_{n < p \leq 2n} e^{\log p} \\
&= \exp\Big(\sum_{n < p \leq 2n} \log p \Big) \\
&= \exp\big(\vartheta(2n)-\vartheta(n)\big).
\end{align*}
Therefore,
\begin{equation*}
\vartheta(2n) - \vartheta(n) \leq 2n \log 2.
\end{equation*}
Set $n = 2^{k-1}$ and deduce
\begin{equation*}
\vartheta(2^k) - \vartheta(2^{k-1}) \leq 2^k \log 2.
\end{equation*}
Since $\vartheta(1) = 0$, a telescoping-series argument and the summation formula for a finite geometric series provide
\begin{align*}
\vartheta(2^k)
&= \vartheta(2^k) - \vartheta(2^0) 
= \sum_{i=1}^k \big (\vartheta(2^i) - \vartheta(2^{i-1}) \big) \\
&\leq \sum_{i=1}^k 2^i \log 2 
< (1+ 2 + 2^2 + \cdots + 2^k) \log 2 \\
&< 2^{k+1} \log 2.
\end{align*}
If $x \geq 1$, then let $2^k \leq x < 2^{k+1}$; that is, let $k = \lfloor \frac{\log x}{\log 2} \rfloor$.
Then
\begin{equation*}
\vartheta(x) \leq \vartheta(2^{k+1}) \leq 2^{k+2} \log 2 = 4 \cdot 2^k \log 2 \leq x (4 \log 2) < 3x
\end{equation*}
since $4 \log 2 \approx 2.7726 < 3$.
\end{proof}

\begin{remark}
The Euler product formula \eqref{eq:EulerProduct}, which requires the fundamental theorem of arithmetic,
and the opening lines of the proof of Chebyshev's lemma are the only portions of our proof of the prime number theorem
that explicitly require number theory.
\end{remark}

\begin{remark}
There are many other ``theta functions,'' some of which arise in 
the context of the Riemann zeta function.  For example, the Jacobi theta function 
$\theta(z)= \sum_{m = - \infty}^{\infty} e^{-\pi m^2 z}$, defined for $\Re z > 0$, is often used in
proving the functional equation \eqref{eq:Functional}.
\end{remark}

\section{The $\Phi$ Function}\label{Section:Phi}
Although we have tried to limit the introduction of new functions,
we must consider
    \begin{equation}\label{eq:PhiDef}
        \Phi(s)= \sum_p \frac{\log p}{p^s},
    \end{equation}
    whose relevance to the prime numbers is evident from its definition.
    If $\Re s \geq \sigma > 1$, then
    \begin{equation*}
        \sum_{p} \left| \frac{\log p}{p^s} \right|
        \leq \sum_{p} \frac{\log p}{p^{\Re s}} 
        \leq \sum_{p} \frac{\log p}{p^{\sigma}}
        < \sum_{n=1}^{\infty} \frac{\log n}{n^{\sigma}}
        < \infty
    \end{equation*}
    by the integral test.
    The Weierstrass $M$-test ensures \eqref{eq:PhiDef} converges uniformly
    on $\Re s \geq \sigma$.  Since the summands in \eqref{eq:PhiDef} are analytic on $\Re s > 1$
    and $\sigma > 1$ was arbitrary, the series \eqref{eq:PhiDef} converges locally
    uniformly on $\Re s > 1$ and hence $\Phi(s)$ is analytic there. 
    For the prime number theorem,
    we need a little more.

    \begin{theorem}\label{Theorem:PhiExtend}
        $\displaystyle\Phi(s) - \frac{1}{s-1}$ is analytic on an open set containing $\Re s \geq 1$.
    \end{theorem}

\begin{proof}
For $\Re s > 1$, \eqref{eq:Thing} tells us
\begin{equation}\label{eq:LogZssum}
\log \zeta(s) 
= \log \Big(\prod_p(1 - p^{-s})^{-1}\Big)
= - \sum_p \log (1 - p^{-s}).
\end{equation}
The inequality \eqref{eq:TheDarnInequality} implies
\begin{equation*}
|1-p^{-s} | \leq \frac{2}{p^{\Re s}},
\end{equation*}
which implies that the convergence in \eqref{eq:LogZssum} is locally uniform on $\Re s > 1$.
Consequently,
we may take the derivative of \eqref{eq:LogZssum} term-by-term and get
\begin{align*}
-\frac{\zeta'(s)}{\zeta(s)}
&= \sum_p \frac{ (\log p) p^{-s}}{1 - p^{-s}} 
= \sum_p (\log p) \left(\frac{ 1}{p^s - 1} \right)\\
&= \sum_p (\log p) \left(\frac{1}{p^s} + \frac{ 1}{p^s(p^s - 1)} \right)
= \sum_p \left(\frac{\log p}{p^s} + \frac{ \log p}{p^s(p^s - 1)} \right)\\
&= \sum_p \frac{\log p}{p^s} + \sum_p  \frac{\log p}{p^s(p^s-1)}
= \Phi(s) + \sum_p  \frac{\log p}{p^s(p^s-1)}.
\end{align*}
If $\Re s \geq \sigma > \frac{1}{2}$, then the limit comparison test and integral test\footnote{Compare
$\sum_{n=2}^{\infty} \frac{\log n}{(n^{\sigma}-1)^2}$ with $\sum_{n=2}^{\infty} \frac{\log n}{n^{2\sigma}}$ and observe that
$\int_2^{\infty} \frac{\log t}{t^{2\sigma}}\,dt <\infty$.}
imply
\begin{equation*}
\sum_p  \left|\frac{\log p}{p^s(p^s-1)}\right|
\leq \sum_{n=2}^{\infty} \frac{\log n}{(n^{\Re s}-1)^2} \\
\leq \sum_{n=2} \frac{\log n}{(n^{\sigma}-1)^2} 
< \infty.
\end{equation*}
The Weierstrass $M$-test ensures that 
$$\sum_p  \frac{\log p}{p^s(p^s-1)}$$ converges locally uniformly on $\Re s > \frac{1}{2}$ and
is analytic there.
Theorem \ref{Theorem:ZetaExtend} implies that
\begin{equation*}
\Phi(s) = -\frac{\zeta'(s)}{\zeta(s)} -   \sum_p  \frac{\log p}{p^s(p^s-1)}
\end{equation*}
extends meromorphically to $\Re s > \frac{1}{2}$
with poles only at $s=1$ and the zeros of $\zeta(s)$.
Theorem \ref{Theorem:ZetaExtend} also yields
\begin{equation*}
\zeta(s) = (s-1)^{-1}Z(s), \qquad Z(1) = 1,
\end{equation*}
in which $Z(s)$ is analytic near $s=1$.
Consequently,
\begin{equation*}
\frac{\zeta'(s)}{\zeta(s)}  = \frac{ -1(s-1)^{-2} Z(s) + (s-1)^{-1} Z'(s)}{(s-1)^{-1}Z(s)}
= - \frac{1}{s-1} + \frac{Z'(s)}{Z(s)}
\end{equation*}
and hence
\begin{equation*}
\Phi(s) - \frac{1}{s-1} =  - \frac{Z'(s)}{Z(s)} -   \sum_p  \frac{\log p}{p^s(p^s-1)},
\end{equation*}
in which the right-hand side is meromorphic on $\Re s > \frac{1}{2}$
with poles only at the zeros of $\zeta(s)$.  
Theorem \ref{Theorem:ZetaZeros} ensures that $\zeta$ has no zeros on $\Re s = 1$,
so the right-hand side extends analytically to some open neighborhood of $\Re s \geq 1$;
see Remark \ref{Remark:Psi}.
\end{proof}

\begin{remark}\label{Remark:Psi}
The zeros of a nonconstant analytic function are isolated, so no bounded sequence of zeta zeros can converge to a point on $\Re s =1$.
Consequently, it is possible to extend $\Phi(s) - (s-1)^{-1}$ a little beyond $\Re s = 1$ in a manner that avoids the zeros of $\zeta(s)$.  It may not be possible to 
do this on a half plane, however.  The Riemann Hypothesis suggests that the half plane $\Re s > \frac{1}{2}$ works, but this remains unproven.
\end{remark}

\section{Laplace Transforms}\label{Section:Laplace}

Laplace transform methods are commonly used to study differential equations 
and often feature prominently in complex-variables texts.  We need only the basic definition
and a simple convergence result.
The following theorem is not stated in the greatest generality possible, but it is sufficient
for our purposes.

\begin{theorem}\label{Theorem:LaplaceAnalytic}
    Let $f:[0,\infty)\to\C$ be piecewise continuous on $[0,a]$ for all $a>0$ and
    \begin{equation*}
    	|f(t)| \leq Ae^{Bt},\qquad \text{for $t \geq 0$}.
    \end{equation*}
    Then the \emph{Laplace transform}
    \begin{equation}\label{eq:Laplace}
        g(z) = \int_0^{\infty} f(t) e^{-zt}\,dt
    \end{equation}
    of $f$ is well defined and analytic on the half plane $\Re z > B$.
\end{theorem}

\begin{proof}
    For $\Re z > B$, the integral \eqref{eq:Laplace} converges by the comparison test since
    \begin{equation*}
    \int_0^{\infty} |f(t) e^{-zt}|\,dt
    \leq  \int_0^{\infty} Ae^{Bt} e^{-t(\Re z)} \, dt
    = A \int_0^{\infty} e^{t(B-\Re z)}\,dt
    =\frac{ A}{\Re z - B} < \infty.
    \end{equation*}
If $\gamma$ is a simple closed curve in $\Re z > B$, then its compactness ensures that there is a $\sigma > B$
such that $\Re z \geq \sigma$ for all $z \in \gamma$.  Thus,
\begin{equation*}
\int_0^{\infty} |f(t) e^{-zt}|\,dt \leq \frac{A}{\sigma - B}
\end{equation*}
is uniformly bounded for $z \in \gamma$.
Fubini's theorem\footnote{The interval $[0,\infty)$ is unbounded and hence the appeal to Fubini's
theorem is more if one uses Riemann integration; see the proof of Theorem \ref{Theorem:ZetaExtend}.}
and Cauchy's theorem yield
\begin{equation*}
\int_{\gamma} g(z)\,dz
= \int_{\gamma} \int_0^{\infty} f(t) e^{-zt}\,dt \,dz
= \int_0^{\infty} f(t) \left(\int_{\gamma} e^{-zt}\,dz\right)\,dt 
= \int_0^{\infty} f(t) \cdot 0 \,dt 
=0.
\end{equation*}
Morera's theorem implies that $g$ is analytic on $\Re z > B$.
\end{proof}

\begin{theorem}[Laplace Representation of $\Phi$]\label{Theorem:LaplacePhi}
For $\Re s > 1$,
\begin{equation}\label{eq:PhiLap}
\frac{\Phi(s) }{s}
= \int_0^{\infty} \vartheta(e^t)e^{-st}\,dt.
\end{equation}
\end{theorem}

\begin{proof}
Recall from Theorem \ref{Theorem:Chebyshev} that $\vartheta(x) \leq 3x$.
Thus, for $\Re s > 2$ 
\begin{equation}\label{eq:AbsConv}
\sum_{n=1}^{\infty} \left| \frac{\vartheta(n-1)}{n^s} \right| \leq 
\sum_{n=1}^{\infty} \left| \frac{\vartheta(n)}{n^s} \right| 
\leq \sum_{n=1}^{\infty} \frac{3n}{n^{\Re s}} = 3\sum_{n=1}^{\infty} \frac{1}{n^{(\Re s)-1}} < \infty.
\end{equation}
Consequently,
\begin{align}
\Phi(s)
&= \sum_p \frac{\log p}{p^s}  && (\text{by \eqref{eq:PhiDef}})\nonumber \\
&= \sum_{n=1}^{\infty} \frac{ \vartheta(n) - \vartheta(n-1)}{n^s} && ( \text{by \eqref{eq:ThetaDef}}) \nonumber \\
&= \sum_{n=1}^{\infty} \frac{\vartheta(n)}{n^s} - \sum_{n=2}^{\infty} \frac{\vartheta(n-1)}{n^s} && ( \text{by \eqref{eq:AbsConv}} ) \nonumber\\
&= \sum_{n=1}^{\infty} \frac{\vartheta(n)}{n^s} - \sum_{n=1}^{\infty} \frac{\vartheta(n)}{(n+1)^s} \nonumber \\
&= \sum_{n=1}^{\infty} \vartheta(n)\left( \frac{1}{n^s} - \frac{1}{(n+1)^s} \right)\nonumber \\
&= \sum_{n=1}^{\infty} \vartheta(n) \left(s\int_n^{n+1} \frac{dx}{x^{s+1}}\right) \nonumber \\
&= s \sum_{n=1}^{\infty}  \int_n^{n+1} \frac{\vartheta(n)\,dx}{x^{s+1}} \nonumber \\
&= s \sum_{n=1}^{\infty}  \int_n^{n+1} \frac{\vartheta(x)\,dx}{x^{s+1}} && 
(\text{$\vartheta(x) = \vartheta(n)$ on $[n,n+1)$})\nonumber \\
&= s\int_1^{\infty} \frac{\vartheta(x)\,dx}{x^{s+1}} \nonumber \\
&= s \int_0^{\infty} \frac{\vartheta(e^t) e^t\,dt}{e^{st+t}} && (\text{$x=e^t$ and $dx= e^t\,dt$})\nonumber \\
&= s \int_0^{\infty} \vartheta(e^t)e^{-st}\,dt .\label{eq:Halfway}
\end{align}
This establishes the desired identity \eqref{eq:PhiLap} for $\Re s > 2$.
Since Theorem \ref{Theorem:Chebyshev} implies $\vartheta(e^t)  \leq 3 e^t$,
Theorem \ref{Theorem:LaplaceAnalytic} (with $A=3$ and $B=1$) ensures that \eqref{eq:Halfway} is analytic on $\Re s > 1$.
On the other hand, $\Phi(s)$ is analytic on $\Re s > 1$ so the
identity principle implies that the desired representation \eqref{eq:PhiLap} holds for $\Re s > 1$.
\end{proof}

\section{Newman's Tauberian Theorem}\label{Section:Newman}

The following theorem is a tour-de-force of undergraduate-level complex analysis.
In what follows, observe that $g$ is the Laplace transform of $f$.  The hypotheses upon $f$
ensure that we will be able to apply the theorem to the Chebyshev theta function.

\begin{theorem}[Newman's Tauberian Theorem]\label{Theorem:Newman}
Let $f:[0,\infty)\to \C$ be a bounded function that is piecewise continuous on $[0,a]$ for each $a >0$.
For $\Re z > 0$, let
\begin{equation*}
g(z) = \int_0^{\infty} f(t) e^{-zt}\,dt
\end{equation*}
and suppose $g$ has an analytic continuation to a neighborhood of $\Re z \geq 0$.  Then 
\begin{equation*}
g(0) = \lim_{T\to\infty} \int_0^T f(t)\,dt.
\end{equation*}
In particular, $\int_0^{\infty} f(t)\,dt$ converges.
\end{theorem}
\begin{proof}
For each $T \in (0,\infty)$, let
\begin{equation}\label{eq:GTdef}
g_T(z) = \int_0^T e^{-zt}f(t)\,dt.
\end{equation}
The proof of Theorem \ref{Theorem:LaplaceAnalytic} ensures that each $g_T(z)$ is an entire function
(see Remark \ref{Remark:Entire} for another approach)
and $g(z)$ is analytic on $\Re z > 0$; see Remark \ref{Remark:Entire}.
We must show 
\begin{equation}\label{eq:g0}
\lim_{T\to\infty}  g_T(0) = g(0).
\end{equation}

\noindent\textsc{Step 1.}
Let $\norm{f}_{\infty} = \sup_{t \geq 0}|f(t)|$, which is finite by assumption.
For $\Re z > 0$,
\begin{align}
|g(z) - g_T(z)|
&= \left| \int_0^{\infty} e^{-zt}f(t)\,dt - \int_0^T e^{-zt}f(t)\,dt \right| 
= \left| \int_T^{\infty} e^{-zt}f(t)\,dt  \right| \nonumber\\
&\leq  \int_T^{\infty} e^{-\Re (zt)}| f(t)|\,dt   
\leq \norm{f}_{\infty} \int_T^{\infty} e^{-t\Re z} \,dt   \nonumber\\
&= \norm{f}_{\infty} \frac{e^{-T \Re z} }{\Re z}. \label{eq:Rz1}
\end{align}

\noindent\textsc{Step 2.}
For $\Re z < 0$,
\begin{align}
|g_T(z)|
&= \left| \int_0^T e^{-zt}f(t)\,dt  \right| 
\leq  \int_0^T e^{-\Re(zt)}|f(t)|\,dt  \nonumber \\
&\leq \norm{f}_{\infty} \int_0^T e^{- t \Re z}\,dt 
\leq \norm{f}_{\infty} \int_{-\infty}^T e^{- t \Re z}\,dt \nonumber\\
&= \norm{f}_{\infty} \frac{e^{-T \Re z}}{|\Re z|}.\label{eq:Rz2}
\end{align}

\noindent\textsc{Step 3.}
Suppose that $g$ has an analytic continuation to an open region $\Omega$
that contains the closed half plane $\Re z \geq 0$. Let $R > 0$ and
let $\delta_R > 0$ be small enough to ensure that $g$ is analytic on an open region
that contains the curve $C_R$ (and its interior) formed by intersecting the circle $|z|=R$
with the vertical line $\Re z = -\delta_R$; see Figure \ref{Figure:CR}.

\begin{figure}
\begin{tikzpicture}
\draw[black, dashed, thin, fill=yellow, opacity=0.15] (0,1) circle (1.6 cm);
\draw[black, dashed, thin, fill=yellow, opacity=0.15] (0,-1) circle (1.6 cm);
\fill[green,opacity=.1] (0,-3) rectangle (3,3);

\draw[<->] (-3,0) -- (3,0) coordinate (xaxis);
\draw[<->] (0,-3) -- (0,3) coordinate (yaxis);


\begin{scope}[very thick,decoration={
    markings,
    mark=between positions 0.05 and 0.95 step 0.15 with {\arrow{>}}}
    ] 
\path[draw,very thick,postaction={decorate},blue] (-1,1.732050808) -- (-1,-1.732050808) arc (-120:120:2);
\end{scope}
\node[blue] at (2,1.25) {\large $C_{R}$};
\node at (-1.4,-0.25) {\small$-\delta_{R}$};
\node at (2.25,-0.25) {\small$R$};
\node at (0.3,2.25) {\small$iR$};
\node at (0.35,-2.25) {\small$-iR$};

\end{tikzpicture}
\caption{The contour $C_R$.  The imaginary line segment $[-iR,iR]$ is compact
and can be covered by finitely many open disks (yellow) upon which $g$ is analytic.  Thus, there is a 
$\delta_R > 0$ such that $g$ is analytic on an open region
that contains the curve $C_R$.}
\label{Figure:CR}
\end{figure}
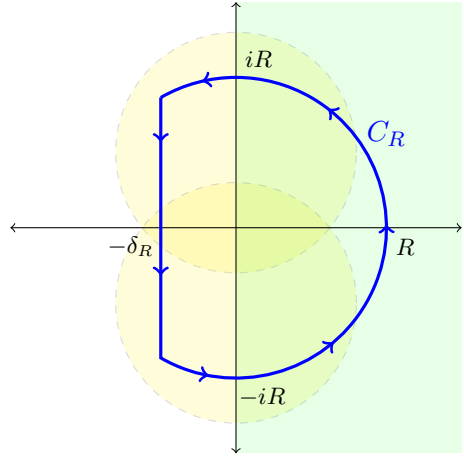

\noindent\textsc{Step 4.}
For each $R > 0$, Cauchy's integral formula implies 
\begin{equation}\label{eq:NewmanTrick}
g_T(0) - g(0)
= \frac{1}{2\pi i} \int_{C_R} \big(g_T(z) - g(z) \big) e^{zT}\left(1 + \frac{z^2}{R^2} \right) \frac{dz}{z}.
\end{equation}
We examine the contributions to this integral over the two curves
\begin{equation*}
C_R^+ = C_R \cap \{z : \Re z \geq 0\}
\qquad \text{and} \qquad 
C_R^- = C_R \cap \{z : \Re z \leq 0\}.
\end{equation*}

\noindent\textsc{Step 5.}
Let us examine the contribution of $C_R^+$ to \eqref{eq:NewmanTrick}.
For $z = Re^{it}$,
\begin{align}
\left|\frac{1}{z}\left(1 + \frac{z^2}{R^2} \right) \right|
&= \left| \frac{1}{z} + \frac{z}{R^2} \right| 
= \left| \frac{1}{Re^{it}} + \frac{Re^{it}}{R^2} \right| \nonumber\\
&= \frac{1}{R^2} | Re^{-it} + Re^{it} | 
= \frac{1}{R^2} | \overline{z} + z | \nonumber\\
&= \frac{2 |\Re z|}{R^2}.\label{eq:CircleEstimate}
\end{align}
For $z \in \C$, 
\begin{equation}\label{eq:MidThing}
|e^{zT}| = e^{T \Re z}
\end{equation}
and hence \eqref{eq:Rz1}, \eqref{eq:CircleEstimate}, and \eqref{eq:MidThing} imply
\begin{align}
&\left| \frac{1}{2\pi i}\int_{C_R^+} \big(g_T(z) - g(z) \big) e^{zT}\left(1 + \frac{z^2}{R^2} \right) \frac{dz}{z} \right| \nonumber\\
&\qquad\leq \frac{1}{2\pi} \underbrace{\left(\norm{f}_{\infty} \frac{e^{-T \Re z} }{\Re z} \right)}_{\text{by \eqref{eq:Rz1}}}
\underbrace{(e^{T \Re z})}_{\text{by \eqref{eq:MidThing}}} 
\underbrace{\left(\frac{2 |\Re z|}{R^2} \right)}_{\text{by \eqref{eq:CircleEstimate}}}(\pi R) \nonumber \\[-5pt]
&\qquad = \frac{\norm{f}_{\infty}}{R}. \label{eq:CR+}
\end{align}

\noindent\textsc{Step 6a.}
We examine the contribution of $C_R^-$ to \eqref{eq:NewmanTrick} in two steps.
Since the integrand in the following integral is analytic in $\Re z < 0$,  we can replace the contour $C_R^-$ with the left-hand side of the circle $|z| = R$ in the computation
\begin{align}
&\left| \frac{1}{2\pi i}\int_{C_R^-} g_T(z) e^{zT}\left(1 + \frac{z^2}{R^2} \right) \frac{dz}{z} \right| \label{eq:IntAnCu} \\
&\qquad=\left| \frac{1}{2\pi i}\int_{\substack{|z|=R\\ \Re z \leq 0}} g_T(z) e^{zT}\left(1 + \frac{z^2}{R^2} \right) \frac{dz}{z} \right| \nonumber \\
&\qquad\leq \frac{1}{2\pi} 
\underbrace{\left(\norm{f}_{\infty} \frac{e^{-T \Re z}}{|\Re z|}\right)}_{\text{by \eqref{eq:Rz2}}}
(e^{T \Re z})
\underbrace{\left(\frac{2 |\Re z|}{R^2}\right)}_{\text{by \eqref{eq:CircleEstimate}}} (\pi R) \nonumber \\
&\qquad=\frac{\norm{f}_{\infty}}{R}; \label{eq:CR-}
\end{align}
see Figure \ref{Figure:Shift}.  

\begin{figure}
\begin{tikzpicture}
\fill[green,opacity=.1] (-3,-3) rectangle (0,3);

\draw[<->] (-3,0) -- (3,0) coordinate (xaxis);
\draw[<->] (0,-3) -- (0,3) coordinate (yaxis);

\begin{scope}[very thick,decoration={
    markings,
    mark=between positions 0.1 and 0.9 step 0.4 with {\arrow{>}}}
    ] 
\path[draw,very thick,postaction={decorate},red] (-1,1.732050808) arc (120:240:2);
\end{scope}

\begin{scope}[very thick,decoration={
    markings,
    mark=between positions 0.29 and 0.7 step 0.4 with {\arrow{>}}}
    ] 
\path[draw,very thick,postaction={decorate},blue, dashed] (-1,1.732050808) -- (-1,-1.732050808);
\end{scope}

\begin{scope}[very thick,decoration={
    markings,
    mark=between positions 0.1 and 0.8 step 0.2 with {\arrow{>}}}
    ] 
\path[draw,very thick,postaction={decorate},blue] (-1,-1.732050808) arc (-120:120:2);
\end{scope}

\node[blue] at (2,1.25) {\large $C_{R}$};
\node at (-1.4,-0.25) {\small$-\delta_{R}$};
\node at (2.25,-0.25) {\small$R$};
\node at (0.3,2.25) {\small$iR$};
\node at (0.35,-2.25) {\small$-iR$};

\end{tikzpicture}
\caption{The integrand in \eqref{eq:IntAnCu} is analytic in $\Re z < 0$.  Cauchy's theorem ensures that the integral
over $C_R^-$ equals the integral over the semicircle $\{ z : |z| = R, \Re z \leq 0\}$.}
\label{Figure:Shift}
\end{figure}
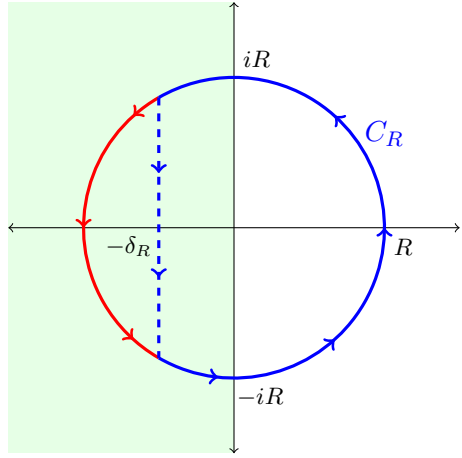

\noindent\textsc{Step 6b.}
Next we focus on the corresponding integral with $g$ in place of $g_T$.
Let $$M = \sup_{z \in C_R^-} |g(z)|,$$ which is finite since $C_R^-$ is compact.  
Since $|z| \geq \delta_R$ for $z \in C_R^-$,
\begin{equation*}
\left| g(z) e^{zT}\left(1 + \frac{z^2}{R^2} \right) \frac{1}{z} \right|
\leq \frac{2M e^{T \Re z} }{\delta_R}.
\end{equation*}
Fix $\epsilon > 0$ and obtain a curve $C_R^-(\epsilon)$
by removing, from the beginning and end of $C_R^-$, two 
arcs each of length $\epsilon\delta_R/(4M)$; see Figure \ref{Figure:Correct}.  Then there is a $\rho >0$ such that
$\Re z < -\rho$ for each $z \in C_{R}^-(\epsilon)$.  Consequently,
\begin{equation*}
\limsup_{T\to\infty}\left| \int_{C_R^-} g(z) e^{zT}\left(1 + \frac{z^2}{R^2} \right) \frac{dz}{z} \right| 
\leq \limsup_{T\to\infty} \bigg( \underbrace{\frac{2M e^{-\rho T} }{\delta_R}\cdot \pi R }_{\text{from $C_R^-(\epsilon)$}} 
+ \underbrace{\frac{2M }{\delta_R} \cdot \frac{2\epsilon \delta_R}{4M} }_{\text{from the two arcs}}  \bigg) = \epsilon.
\end{equation*}
Since $\epsilon > 0$ was arbitrary,
\begin{equation}\label{eq:CRF}
\limsup_{T\to\infty}\left| \int_{C_R^-} g(z) e^{zT}\left(1 + \frac{z^2}{R^2} \right) \frac{dz}{z} \right| =0.
\end{equation}

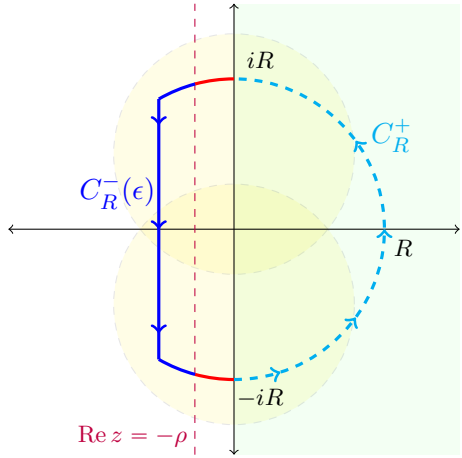
\begin{figure}
\begin{tikzpicture}
\draw[black, dashed, thin, fill=yellow, opacity=0.1] (0,1) circle (1.6 cm);
\draw[black, dashed, thin, fill=yellow, opacity=0.1] (0,-1) circle (1.6 cm);
\fill[green,opacity=.05] (0,-3) rectangle (3,3);

\draw[<->] (-3,0) -- (3,0) coordinate (xaxis);
\draw[<->] (0,-3) -- (0,3) coordinate (yaxis);
\draw[dashed, purple] (-0.52,-3) -- (-0.52,3) coordinate (yaxis);

\begin{scope}[very thick,decoration={
    markings,
    mark=between positions 0.1 and 0.9 step 0.4 with {\arrow{>}}}
    ] 
\path[draw,very thick,postaction={decorate},blue] (-1,1.732050808) -- (-1,-1.732050808);
\end{scope}

\begin{scope}[very thick,decoration={
    markings,
    mark=between positions 0.1 and 0.8 step 0.2 with {\arrow{>}}}
    ] 
\path[draw,very thick,postaction={decorate},cyan, dashed] (0,-2) arc (-90:90:2);
\end{scope}

\path[draw,very thick,red] (0,2) arc (90:105:2);
\path[draw,very thick,red] (0,-2) arc (-90:-105:2);
\path[draw,very thick,blue] (-1,1.732050808)  arc (120:105:2);
\path[draw,very thick,blue] (-1,-1.732050808)  arc (-120:-105:2);

\node[cyan] at (2.1,1.25) {\large $C_{R}^+$};
\node[blue] at (-1.55,0.5) {\large $C_{R}^-(\epsilon)$};
\node at (2.25,-0.25) {\small$R$};
\node at (0.35,2.25) {\small$iR$};
\node[purple] at (-1.35,-2.75) {\small$\Re z = -\rho$};
\node at (0.35,-2.25) {\small$-iR$};

\end{tikzpicture}
\caption{$C_{R}^-(\epsilon)$ is obtained from $C_R^-$ by removing two segments (red) each of length $\epsilon\delta_R/(4M)$.
There is a $\rho > 0$ such that $\Re z < -\rho$ for each $z \in C_{R}^-(\epsilon)$.
}
\label{Figure:Correct}
\end{figure}

\noindent\textsc{Step 7.}
For each fixed $R > 0$,
\begin{align*}
&\limsup_{T\to 0} |g_T(0) - g(0)| \\
&\qquad=\limsup_{T\to 0} \left|  \frac{1}{2\pi i} \int_{C_R} \big(g_T(z) - g(z) \big) e^{zT}\left(1 + \frac{z^2}{R^2} \right) \frac{dz}{z} \right| && ( \text{by \eqref{eq:NewmanTrick}} )\\
&\qquad \leq \underbrace{\frac{\norm{f}_{\infty}}{R}}_{\text{from $C_R^+$}} + 
\underbrace{\left(\frac{\norm{f}_{\infty}}{R} + 0\right) }_{\text{from $C_R^-$}}  && (\text{by \eqref{eq:CR+}, \eqref{eq:CR-}, \eqref{eq:CRF}})\\
&\qquad= \frac{2 \norm{f}_{\infty}}{R}.
\end{align*}
Since $R >0$ was arbitrary, 
\begin{equation*}
\limsup_{T\to \infty} |g_T(0) - g(0)|  = 0;
\end{equation*}
that is,
$\lim_{T\to\infty} g_T(0) = g(0)$.
\end{proof}

\begin{remark}
A ``Tauberian theorem'' is a result in which a convergence result
is deduced from a weaker convergence result and an additional hypothesis.
The phrase originates in the work of G.H.~Hardy and J.E.~Littlewood, who coined the term
in honor of A.~Tauber.
\end{remark}

\begin{remark}\label{Remark:Entire}
To see that $g_T(z)$ entire, first note that since we are integrating over $[0,T]$ there are no convergence issues.  We may let $\gamma$ be any simple closed curve in $\C$ when we mimic the proof of Theorem \ref{Theorem:LaplaceAnalytic}.  Another approach is to expand $e^{-zt}$ as a power series and use the uniform
convergence of the series on $[0,T]$ to exchange the order of sum and integral.  This yields a power series
expansion of $g_T(z)$ with infinite radius of convergence.  Here are the details.
    Fix $T > 0$ and let 
    \begin{equation*}
        M = \sup_{0\leq t \leq T}|f(t)|,
    \end{equation*}
    which is finite since $[0,T]$ is compact and $f$ is piecewise continuous
    (a piecewise-continuous function has at most finitely many discontinuities, all of which are 
    jump discontinuities).  Then
    \begin{equation*}
        c_n = \int_0^T f(t)  t^n\,dt
        \quad \text{satisfies} \quad 
        |c_n|\leq  \frac{MT^{n+1}}{n+1}.
    \end{equation*}
    Since $e^z$ is entire, its power series representation converges uniformly on
    $[0,T]$.  Thus,
    \begin{align*}
        g_T(z) 
        &= \int_0^T f(t)e^{-zt}\,dt   
        = \int_0^T f(t) \bigg(\sum_{n=0}^{\infty} \frac{(-zt)^n}{n!} \bigg)\,dt \\
        &= \sum_{n=0}^{\infty} \frac{(-1)^n z^n}{n!} \int_0^T f(t)  t^n\,dt 
        = \sum_{n=0}^{\infty} \frac{(-1)^n  c_n}{n!}z^n 
    \end{align*}
    defines an entire function since its radius of convergence is the reciprocal of
    \begin{equation*}
        \limsup_{n\to\infty} \left| \frac{(-1)^n  c_n}{n!}\right|^{\frac{1}{n}}
        \leq  \limsup_{n\to\infty} \frac{M^{\frac{1}{n}} T^{\frac{n+1}{n}}}{ (n+1)^{\frac{1}{n}}(n!)^{\frac{1}{n}}}
        = \frac{1\cdot T}{1 \cdot \infty} = 0
    \end{equation*}
    by the Cauchy--Hadamard formula.
\end{remark}

\begin{remark}
Step 6b is more complicated than in most presentations because we are using 
the Riemann integral (for the sake of accessibility) instead of the Lebesgue integral.  
The statement \eqref{eq:CRF} follows immediately from the Fatou--Lebesgue theorem in Lebesgue theory; see the proof in \cite{Simon}.
Riemann integration theory cannot prove \eqref{eq:CRF} directly since the integrand does not convergence
uniformly to zero on $C_R^-$.
\end{remark}

\section{An Improper Integral}\label{Section:Improper}
    Things come together in the following lemma.  
    We have done most of the difficult work already; the proof of Lemma \ref{Lemma:Converge}
    amounts to a series of strategic applications of existing results.
    It requires Chebyshev's estimate for $\vartheta(x)$ (Theorem \ref{Theorem:Chebyshev}),
    the analytic continuation of 
    $\Phi(s) - (s-1)^{-1}$ to an open neighborhood of $\Re s \geq 1$ (Theorem \ref{Theorem:PhiExtend}),
    the Laplace-transform representation of $\Phi(s)$ (Theorem \ref{Theorem:LaplacePhi}), and 
    Newman's theorem (Theorem \ref{Theorem:Newman}).
    
    \begin{lemma}\label{Lemma:Converge}
        $\displaystyle \int_1^{\infty} \frac{\vartheta(x)-x}{x^2}\,dx$ converges.
    \end{lemma}

    \begin{proof}
        Define $f:[0,\infty)\to \C$ by
        \begin{equation*}
            f(t) = \vartheta(e^t)e^{-t}-1
        \end{equation*}
        and observe that it is piecewise continuous on $[0,a]$ for all $a>0$ and
        \begin{equation*}
            |f(t)| \leq |\vartheta(e^t)| e^{-t} + 1
            \leq 4
        \end{equation*}
        for all $t \geq 0$
        by Theorem \ref{Theorem:Chebyshev}.  
        Then Theorem \ref{Theorem:LaplaceAnalytic} with $A=4$ and $B=0$ ensures that the Laplace transform
        of $f$ is analytic on $\Re z > 0$.  Consequently, for $\Re z > 0$
        \begin{align}
            \int_0^{\infty} f(t) e^{-zt}\,dt
            &=\int_0^{\infty} \big( \vartheta(e^t)e^{-t}-1 \big) e^{-zt}\,dt \nonumber\\
            &=  \int_0^{\infty} \big(\vartheta(e^t) e^{-(z+1)t} - e^{-zt}\big)\,dt \nonumber\\
            &=  \int_0^{\infty}\vartheta(e^t)e^{-(z+1)t} \,dt - \int_0^{\infty} e^{-zt}\,dt \label{eq:ThisStepThing}\\
            &=  \int_0^{\infty}\vartheta(e^t)e^{-(z+1)t} \,dt - \frac{1}{z} \nonumber\\
            &= \frac{\Phi(z+1)}{z+1} - \frac{1}{z} && ( \text{by Theorem \ref{Theorem:LaplacePhi}}). \nonumber
        \end{align}
        Let $z = s-1$ and note that
        Theorem \ref{Theorem:PhiExtend} implies that
        \begin{equation*}
            g(z) = \frac{\Phi(z+1)}{z+1} - \frac{1}{z} = \frac{\Phi(s)}{s} - \frac{1}{s-1}
        \end{equation*}
        extends analytically to an open neighborhood of $\Re s \geq 1$;
        that is, to an open neighborhood of the closed half plane $\Re z \geq 0$.
        Theorem \ref{Theorem:Newman} ensures that the improper integral
        \begin{align*}
            \int_0^{\infty} f(t)\,dt
            &= \int_0^{\infty} \big(\vartheta(e^t)e^{-t}-1\big)\,dt \\
            &= \int_1^{\infty} \bigg(\frac{\vartheta(x)}{x}-1\bigg)\,\frac{dx}{x} &&(\text{$x = e^t$ and $dx = e^t\,dt$})\\
            &= \int_1^{\infty} \frac{\vartheta(x)-x}{x^2}\,dx
        \end{align*}
        converges.
    \end{proof}
    
    \begin{remark}
        Newman's theorem implies that the improper integral in Lemma \ref{Lemma:Converge} equals $g(0)$ although this is not necessary for our purposes.
    \end{remark}
    
    \begin{remark}
        Since $|\vartheta(e^t)| \leq 3e^t$ by Theorem \ref{Theorem:Chebyshev},
        the first improper integral \eqref{eq:ThisStepThing}
        converges and defines an analytic function on $\Re z > 0$ 
        by Theorem \ref{Theorem:LaplaceAnalytic} with $A = 3$ and $B = 1$.
        We did not mention this in the proof of Lemma \ref{Lemma:Converge}
        because the convergence of the integral is already guaranteed by the convergence of 
        $\int_0^{\infty} f(t) e^{-zt}\,dt$ and
        $\int_0^{\infty} e^{-zt}\,dt$.
    \end{remark}

\section{Asymptotic Behavior of $\vartheta(x)$}\label{Section:ThetaAsymptotic}
A major ingredient in the proof of the prime number theorem is the following asymptotic
statement.  Students must be comfortable with limits superior and inferior after this point; these concepts are used
frequently throughout what follows.

\begin{theorem}\label{Theorem:Thetax}
    $\vartheta(x) \sim x$.
\end{theorem}

\begin{proof}
    Observe that
    \begin{equation}\label{eq:TailZero}
 	\underbrace{\int_1^{\infty} \frac{\vartheta(t)-t}{t^2}\,dt \quad\text{exists}}_{\text{by Lemma \ref{Lemma:Converge}}}
    \qquad\implies\qquad
    \lim_{x\to\infty} \underbrace{\int_x^{\infty} \frac{\vartheta(t)-t}{t^2}\,dt }_{I(x)} = 0.
    \end{equation}
    
    \noindent\textsc{Step 1.} Suppose toward a contradiction that 
    \begin{equation*}
    \limsup_{x\to\infty} \frac{\vartheta(x)}{x}  > 1,
    \quad\text{and let}\quad
    \limsup_{x\to\infty} \frac{\vartheta(x)}{x}  > \alpha > 1.
    \end{equation*}
    Then there are arbitrarily large $x > 1$ such that $\vartheta(x) > \alpha x$.  For such ``bad'' $x$,
    \begin{align*}
    I(\alpha x) - I(x)
    &= \int_x^{\alpha x} \frac{\vartheta(t) - t}{t^2}\,dt 
    \geq \int_x^{\alpha x} \frac{\alpha x - t}{t^2}\,dt && (\substack{\text{$\vartheta(x)  > \alpha x$ and}\\ \text{$\vartheta$ is increasing}})\\
    &= \int_1^{\alpha} \frac{\alpha x - xu}{x^2u^2}x\,du
    = \int_1^{\alpha} \frac{\alpha - u}{u^2}\,du     && (t = xu,\, dt = x\,du)\\
    &= \alpha-1 - \log \alpha 
    > 0.
    \end{align*}
    Since
    \begin{equation*}
    \liminf_{\substack{x\to\infty\\ \text{$x$ bad}}} \big(  I(\alpha x) - I(x) \big)  > 0
    \end{equation*}
    contradicts \eqref{eq:TailZero}, we conclude
    \begin{equation*}
    \limsup_{x\to\infty} \frac{\vartheta(x)}{x} \leq  1.
    \end{equation*}
    
    \noindent\textsc{Step 2.}  This is similar to the first step.
    Suppose toward a contradiction that 
    \begin{equation*}
        \liminf_{x\to\infty} \frac{\vartheta(x)}{x}  < 1,
            \quad\text{and let}\quad
         \liminf_{x\to\infty} \frac{\vartheta(x)}{x}  < \beta < 1;
    \end{equation*}
    the limit inferior is nonnegative since $\vartheta(x)$ is nonnegative.   
    Then there are arbitrarily large $x > 1$ such that $\vartheta(x)  < \beta x$.  For such ``bad'' $x$,
    \begin{align*}
    I(x) - I(\beta x) 
    &= \int^x_{\beta x} \frac{\vartheta(t) - t}{t^2}\,dt 
    \leq \int^x_{\beta x}\frac{\beta x - t}{t^2}\,dt &&  (\substack{\text{$\vartheta(x)  < \beta x$ and}\\ \text{$\vartheta$ is increasing}})\\
    &= \int_{\beta}^1 \frac{\beta x - xu}{x^2u^2}x\,du
    = \int_{\beta}^1 \frac{\beta - u}{u^2}\,du && (t = xu,\, dt = x\,du)\\
    &=1 - \beta + \log \beta
    < 0.
    \end{align*}
    Since
    \begin{equation*}
    \liminf_{\substack{x\to\infty\\ \text{$x$ bad}}} \big( I(x) - I(\beta x)  \big)  < 0
    \end{equation*}
    contradicts \eqref{eq:TailZero}, we conclude
    \begin{equation*}
    \liminf_{x\to\infty} \frac{\vartheta(x)}{x} \geq  1.
    \end{equation*}
    
    \noindent\textsc{Step 3.}
    Since
    \begin{equation*}
    \limsup_{x\to\infty} \frac{\vartheta(x)}{x} \leq 1
    \qquad \text{and} \qquad
    \liminf_{x\to\infty} \frac{\vartheta(x)}{x} \geq 1,
    \end{equation*}
    it follows that  $\lim_{x\to\infty} \vartheta(x)/x = 1$; that is, $\vartheta(x) \sim x$.
\end{proof}

\begin{remark}
Let $f(x) = x - 1 - \log x$ for $x>0$.
Then $f'(x) = 1-1/x$ and $f''(x) = 1/x^2$, so
$f$ is strictly positive on $(0,1)$ and $(1,\infty)$; see Figure \ref{Figure:LogThing}.
This ensures the positivity of $\alpha - 1 - \log \alpha$ for $\alpha > 1$ and the negativity of 
$1 - \beta + \log \beta$ for $\beta \in (0,1)$.
\end{remark}

\begin{figure}
        \includegraphics[width=0.5\textwidth]{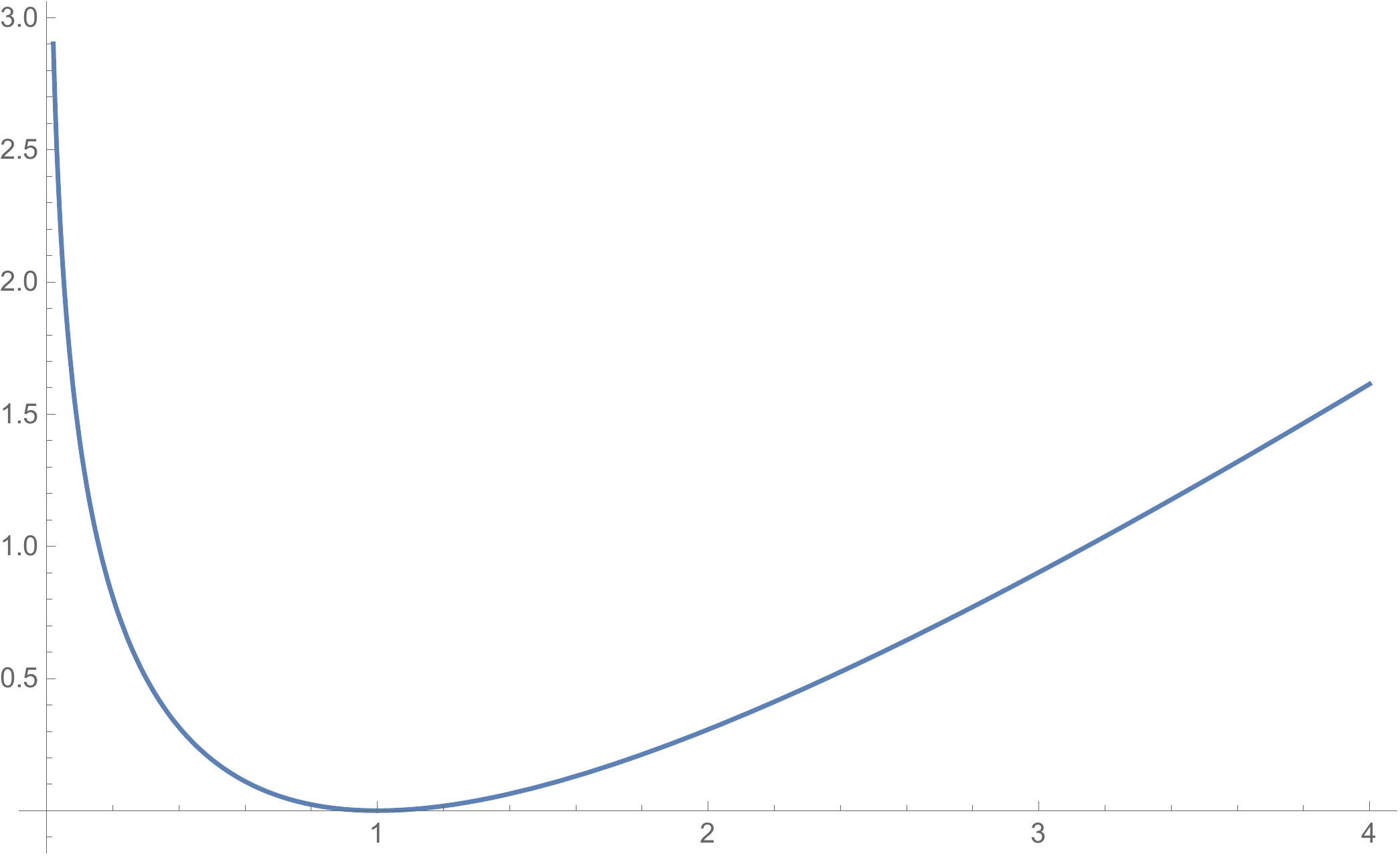}
        \caption{Graph of $f(x) = x - 1 - \log x$.}
        \label{Figure:LogThing}
\end{figure}

\begin{remark}
One can show that $\pi(x) \sim x / \log x$ implies $\vartheta(x) \sim x$, 
although this is not necessary for our purposes.  In light of Theorem \ref{Theorem:Finale} below, this implication 
shows that $\pi(x) \sim x/\log x$ is equivalent to $\vartheta(x) \sim x$.
\end{remark}

\section{Completion of the Proof}\label{Section:Conclusion}
At long last we are ready to complete the proof of the prime number theorem.
We break the conclusion of the proof into three short steps.

\begin{theorem}[Prime Number Theorem]\label{Theorem:Finale}
$\displaystyle \pi(x) \sim \frac{x}{\log x}$.
\end{theorem}

\begin{proof}
    Recall from Theorem \ref{Theorem:Thetax} 
    that $\vartheta(x) \sim x$; that is, $\lim_{x\to\infty} \vartheta(x)/x =1$.
    
    \noindent\textsc{Step 1.} Since
    \begin{equation*}
    \vartheta(x) = \sum_{p\leq x} \log p \leq \sum_{p \leq x} \log x = (\log x) \sum_{p\leq x} 1 = \pi(x) \log x,
    \end{equation*}
    it follows that
    \begin{equation*}
    1 = \lim_{x\to\infty} \frac{ \vartheta(x)}{x} = \liminf_{x\to\infty} \frac{\vartheta(x)}{x}\leq \liminf_{x\to\infty}\frac{\pi(x) \log x}{x} .
    \end{equation*}
    
    \noindent \textsc{Step 2.} For any $\epsilon > 0$,
    \begin{align*}
    \vartheta(x)
    &= \sum_{p \leq x} \log p 
    \geq \sum_{x^{1-\epsilon} < p \leq x} \log p  \\
    &\geq \sum_{x^{1-\epsilon} < p \leq x} \log (x^{1-\epsilon} ) 
    =  \log (x^{1-\epsilon} ) \sum_{x^{1-\epsilon} < p \leq x} 1\\
    &=  (1-\epsilon)(\log x) \bigg(\sum_{p \leq x} 1 - \sum_{p \leq x^{1-\epsilon}} 1 \bigg)\\
    &\geq (1-\epsilon)\big( \pi(x) - x^{1-\epsilon} \big)\log x.
    \end{align*}
    Therefore,
    \begin{align*}
    1
    &= \lim_{x\to\infty} \frac{\vartheta(x)}{x} = \limsup_{x\to\infty} \frac{\vartheta(x)}{x} \\
    &\geq \limsup_{x\to\infty}\left(\frac{(1-\epsilon) \big( \pi(x) - x^{1-\epsilon} \big)\log x}{x}\right) \\
    &= (1-\epsilon)\limsup_{x\to\infty} \left( \frac{\pi(x)\log x}{x} - \frac{\log x}{x^{\epsilon}} \right) \\
    &= (1-\epsilon)\limsup_{x\to\infty}  \frac{\pi(x)\log x}{x} - (1-\epsilon)\lim_{x\to\infty}\frac{\log x}{x^{\epsilon}}  \\
    &= (1-\epsilon)\limsup_{x\to\infty}  \frac{\pi(x)\log x}{x}.
    \end{align*}
    Since $\epsilon > 0$ was arbitrary, 
    \begin{equation*}
    \limsup_{x\to\infty}  \frac{\pi(x)\log x}{x} \leq 1.
    \end{equation*}
    
    \noindent\textsc{Step 3.}
    Since 
    \begin{equation*}
    1 \leq \liminf_{x\to\infty}  \frac{\pi(x)}{x/\log x} \leq 
    \limsup_{x\to\infty}  \frac{\pi(x)}{x/\log x} \leq 1,
    \end{equation*}
    we obtain
    \begin{equation*}
    \lim_{x\to\infty} \frac{\pi(x)}{x/\log x} = 1. 
    \end{equation*}    
    This concludes the proof of the prime number theorem.
\end{proof}

It is probably best not to drag things out at this point.  Nothing can compete with finishing off one of the major
theorems in mathematics.  
After coming this far, the reader should be convinced that the proof of the prime number theorem,
as presented here, is largely a theorem of complex analysis (obviously this is a biased perspective based upon
our choice of proof).  Nevertheless, we hope that the reader is convinced that a proof of the prime number theorem
can function as an excellent capstone for a course in complex analysis.

\bibliographystyle{amsplain} 
\bibliography{PNT}
\end{document}